\documentclass[11pt]{amsart}
\usepackage{lmodern}
\usepackage{amsmath, amsthm, amssymb, amsfonts}
\usepackage[normalem]{ulem}

\usepackage{hyperref}

\usepackage{mathrsfs}

\usepackage{verbatim} 
\usepackage{longtable}

\usepackage{mathtools}

\usepackage{tikz}
\usetikzlibrary{decorations.pathmorphing}
\tikzset{snake it/.style={decorate, decoration=snake}}

\usepackage{caption}

\usepackage{tikz-cd}
\usetikzlibrary{arrows}
\input xy
\xyoption{all}

\theoremstyle{plain}
\newtheorem{thm}{Theorem}[section]
\newtheorem{cor}[thm]{Corollary}
\newtheorem{lem}[thm]{Lemma}
\newtheorem{prop}[thm]{Proposition}

\newtheorem{defn}[thm]{Definition}

\theoremstyle{definition}

\theoremstyle{remark}
\newtheorem{rmk}[thm]{Remark}

\newcommand{\BG}{{\mathbb{G}}}

\newcommand{\BP}{{\mathbb{P}}}
\newcommand{\BQ}{{\mathbb{Q}}}

\newcommand{\BZ}{{\mathbb{Z}}}

\newcommand{\CB}{{\mathcal B}}
\newcommand{\CC}{{\mathcal C}}
\newcommand{\CD}{{\mathcal D}}

\newcommand{\CF}{{\mathcal F}}
\newcommand{\CG}{{\mathcal G}}

\newcommand{\CI}{{\mathcal I}}

\newcommand{\CM}{{\mathcal M}}

\newcommand{\CO}{{\mathcal O}}
\newcommand{\CP}{{\mathcal P}}

\newcommand{\CR}{{\mathcal R}}
\newcommand{\CS}{{\mathcal S}}

\newcommand{\Fc}{{\mathfrak{c}}}

\newcommand{\Ft}{{\mathfrak{t}}}

\newcommand{\sslash}{\mathbin{/\mkern-6mu/}}

\DeclareFontFamily{OT1}{rsfs}{}
\DeclareFontShape{OT1}{rsfs}{n}{it}{<-> rsfs10}{}
\DeclareMathAlphabet{\curly}{OT1}{rsfs}{n}{it}

\newcommand\Spec{\operatorname{Spec}}

\usepackage{tikz}
\usepackage{lmodern}
\usetikzlibrary{decorations.pathmorphing}

\makeatletter
\def\SL@margtext#1{%
  \ifmmode
    \xdef\SL@labelname{\SL@prlabelname{#1}}%
  \else
    \setbox\@tempboxa=\vtop to 0pt{\vss
  \@for\@tmpa:=#1\do{%
      \hbox to \hsize{\expandafter\SL@lrtext\expandafter{\@tmpa}}}\vss}%
    \dp\@tempboxa\z@
    \ifvmode
      \@tempdima=\prevdepth
      \nointerlineskip\box\@tempboxa\nobreak
      \prevdepth=\@tempdima
    \else
      \vadjust{\box\@tempboxa\nobreak}%
    \fi
  \fi
}
\let\SL@setlabel\SL@margtext
\makeatother

\hypersetup{pdfborder=0 0 0} 

\newcommand{\pcs}{ \,^{\frak p}\!{\mathcal H}   }

\newcommand{\oql}{\overline{\BQ}_{\ell}}

\usepackage{color} 
\definecolor{darkgreen}{rgb}{0.0, 0.7, 0.0}
\newenvironment{??}{\noindent \color{darkgreen}{\bf ???:} \footnotesize}{}
\definecolor{cyan}{cmyk}{1,0,0,0}

\newcommand{\bdg}{\begin{dg}}

\begin{document}
\title{The de Rham stack  and the variety of very good splittings
of a curve}
\date{\today}

\author[M. A. de Cataldo]{Mark Andrea A. de Cataldo}
\address{Stony Brook University}
\email{mark.decataldo@stonybrook.edu}

\author[M. Groechenig]{Michael Groechenig}
\address{University of Toronto}
\email{michael.groechenig@utoronto.ca}

\author[S. Zhang]{Siqing Zhang}
\address{Stony Brook University}
\email{siqing.zhang@stonybrook.edu}

\maketitle

\begin{abstract}
    The stack of relative splittings of a special Azumaya algebra plays a key role in the Non-Abelian Hodge Theory for curves in positive characteristics. 
    In this paper, we define and study an open substack consisting of the so-called very good splittings.
    We  show that, when using very good splittings, the Non-Abelian Hodge isomorphism  preserves the semistable loci on the Dolbeault and the de Rham sides.
    We also show that the stack of very good splittings admits a quasi-projective tame moduli space.
    As a consequence, we show that the derived pushforwards of the intersection complexes by the Hitchin and the de Rham-Hitchin morphisms are isomorphic and they have isomorphic perverse cohomology sheaves. 
\end{abstract}

\setcounter{tocdepth}{1} 

\tableofcontents
\setcounter{section}{0}

\section{Introduction}\label{intro}

It is well-known, by the work of Simpson, Corlette and others, that semistable
Higgs bundles and semistable flat connections over a complex projective manifold are intimately related, see \cite{Corlette}, \cite{Donaldson}, \cite{Hit}, \cite{SimpICM}.
This correspondence, known as the Non-Abelian Hodge Theorem (NAHT), admits several variants and generalizations, see  \cite[\S2]{HT1}, \cite{Konno}, \cite{Alfaya}, including Mochizuki's generalization to projective manifolds endowed with simple normal crossings divisors \cite{Mochizuki}.
 
We consider the situation over an algebraically closed field of positive characteristic $p>0$
and we restrict our attention to nonsingular projective integral curves.
In this case, there is a kind of Frobenius-twisted NAHT:
fix the rank, but not the degree;
the algebraic stack  $h_{\rm dR}:{\mathcal M}_{\rm dR}(C) \to A(C^{(1)})$ of rank $r$ flat connections on the curve $C$
(necessarily of degree a multiple of the characteristic $p$) is \'etale locally equivalent,
over the base $A(C^{(1)})$ of the de Rham-Hitchin morphism $h_{\rm dR}$, to the algebraic stack 
$h:\mathcal{M}_{\rm Dol}(C^{(1)}) \to A(C^{(1)})$ of Higgs bundles
over the Frobenius twist $C^{(1)}$ of the curve $C,$ where $h$ is the associated Hitchin morphism. The base $A(C^{(1)})$ parameterizes the  degree $r$ spectral curves inside
the cotangent bundle $T^*{C^{(1)}}.$  This NAHT has been studied in \cite{OV}
(which also generalizes the Deligne-Illusie Decomposition Theorem as in \cite{DeIl}), and in \cite{BB07}, \cite{Groch}, \cite{ChZh}, (which also have applications to the Geometric Langlands Program in characteristic $p$). 

One key object in this set-up  is the  algebraic stack $\mathcal S\to A(C^{(1)})$ of relative splittings of the Azumaya algebra
$\mathcal D$ on the cotangent bundle $T^*C^{(1)}.$
It classifies the splittings, which are suitable rank $p$ vector bundles, of the Azumaya algebra restricted to the degree $r$
spectral curves
in $A(C^{(1)})$.
The stack $\mathcal S$ is a torsor for the Picard stack $\mathcal P\to A(C^{(1)})$ of line bundles on the spectral curves.

The twisted NAHT  can be stated as follows:
there is a natural isomorphism of stacks: 
$\mathcal{S}  \times^{\mathcal P}  \mathcal{M}_{\rm Dol}(C)  
\stackrel{\simeq}\to
\mathcal{M}_{\rm dR}(C^{(1)}).$

This isomorphism does not respect degrees, slopes, nor (semi)stability. It is therefore natural to look for a better object than $\mathcal S.$

The paper \cite[Definition 3.44]{Groch} introduced the stack of good splittings $\mathcal S^0,$ i.e. the  connected component  of the stack of splittings of a certain fixed total degree 
inside the stack of splittings 
$\mathcal S$ and it erroneously claimed \cite[Lemma 3.46]{Groch} that
it preserves semistability. This paper also serves as an erratum to this claim, which is replaced here by the isomorphism (\ref{main iso}).
See Remark \ref{mainmis} for a discussion.

While the algebraic moduli stacks $\mathcal M$ of fixed degree and rank admit quasi-projective adequate (tame, resp.) moduli schemes
of semistable (stable, resp.) objects,  the substacks of $\mathcal S$ and $\mathcal P$ of fixed degrees are inherently non-separated and of infinite type: for example,  the relative Picard stack of  fixed degree line bundles on a relative curve with a reducible special fiber yields a non-separated moduli scheme of infinite type, see \cite[p.210]{BLR}.
The stack $\mathcal P^o$ of multidegree zero line bundles on spectral curves admits a
tame quasi-projective moduli space $P^o$, which is in fact a smooth group scheme
over $A(C^{(1)})$ with geometrically connected fibers.
The stack $\mathcal S$ is not a trival $\mathcal P$-torsor, so it is a priori unclear
if there should be an analogue of $\mathcal P^o,$ denoted by   $\mathcal S^o,$  contained in  $\mathcal S.$

In this paper, we identify such an open substack $\mathcal S^o$ of $\mathcal S$ in Definition \ref{defvgs} and we name it the stack of very good splittings.
We show that $\mathcal S^o$ is naturally a $\mathcal P^o$-torsor in Lemma \ref{poso}. 
We show in Theorem \ref{vgsch} that  $\mathcal S^o$  admits a  good and tame quasi-projective moduli space $S^o$
which is a $P^o$-torsor.
We show in Theorem \ref{sseq} that the restriction $\mathcal S^o \times^{\mathcal P^o}
\mathcal M_{\rm Dol} (C) \simeq \mathcal M_{\rm dR} (C^{(1)})$
of 
$\mathcal S \times^{\mathcal P}
\mathcal M_{\rm Dol} (C) \simeq \mathcal M_{\rm dR} (C^{(1)})$
preserves (semi)stability because it sends objects of degree $d$ to objects of
degree $dp,$ so that it preserves slopes, thus yielding an $A(C^{(1)})$-isomorphism
at the level of moduli spaces of (semi)stable objects (Theorem \ref{schois}):
\begin{equation}\label{main iso}
S^o\times^{P^o} M^{ss}_{\rm Dol} (C^{(1)},d) \simeq_{A(C^{(1)})} M^{ss}_{\rm dR} (C,dp).
\end{equation}

Note that there is no natural identification of Higgs and de Rham objects, they are related 
by splittings of the Azumaya algebra on spectral curves, for which there is no natural
global choice, for $S^o$ is not a trivial $P^o$-torsor, see Lemma \ref{notriv}.

Nevertheless, by exploiting the geometric connectedness of the fibers
of $P^o/A(C^{(1)}),$ we prove a cohomological identity as an application of the isomorphism
(\ref{main iso}), namely we prove in Theorem \ref{isoic} that for every $d\in \mathbb Z$ there are distinguished isomorphisms
of  graded $\oql$-vector spaces
between the intersection cohomology groups of the two moduli spaces:
\begin{equation}\label{applic}
I\!H^*(M_{\rm Dol}^{ss}(C^{(1)}, d)) \simeq I\!H^*(M_{\rm dR}^{ss}(C, dp)).
\end{equation} 

We produce such isomorphisms by proving first that there is a
canonical identification  of the  perverse cohomology sheaves for the de Rham-Hitchin and Hitchin morphisms (Theorem \ref{isodec}).
At this stage, according to the Decomposition Theorem \cite{BBD}, there are many possible isomorphisms of the form (\ref{applic}).
In Remark \ref{dspl}, based on the choices of relative ample classes, we assemble the canonical isomorphisms of perverse cohomology sheaves of the direct image complexes into five distinguished isomorphisms (\ref{applic}) via the Decomposition Theorem and the Deligne splitting \cite{D}, \cite[Thm. 1.1.1]{dC13}.
 We ignore if any of these splittings produces,
when the moduli spaces are nonsingular, an isomorphism of cohomology rings.

An isomorphism analogous to (\ref{main iso}) was claimed in \cite[Lemma 3.46]{Groch}
by using another  open algebraic substack $\mathcal S^0$ of $\mathcal S$
of splittings of a suitable total degree on degree $r$ spectral curves. 
This algebraic stack $\mathcal S^0$ is akin to the Picard stack $\mathcal P^0$
of line bundles of total degree zero on degree $r$ spectral curves,
and is in fact a torsor for it.
It strictly contains $\mathcal S^o$ as an open substack. However, as simple examples in Remark \ref{mainmis} show,
it is not true that
using $\mathcal S^0$ in place of $\mathcal S^o$ preserves (semi)stability. Finally, the algebraic stack $\mathcal S^0$
is of infinite type just like $\mathcal P^0$  and it gives rise to
a non-separated moduli algebraic space.

{\bf Acknowledgments.} We heartfully thank Andres J. Fernandez Herrero 
for patiently explaining some aspects of the theory of algebraic stacks. We thank Antonio Rapagnetta for very useful discussions on torsion free sheaves. We thank Davesh Maulik for asking about the  possibility of having a decomposition theorem along the lines of Theorem \ref{isoic}.
The first and third-named authors have been partially supported by NSF grants
DMS-1901975 and 
DMS-2200492. The first-named author has been partially supoorted by a Simons Fellowship in Mathematics Award n. 672936. The second-named author was supported by an NSERC Discovery Grant (RGPIN-2019-05264) and an Alfred P. Sloan fellowship.
Mark Andrea de Cataldo dedicates this paper to the memory of Tony, with friendship and love.

\section{The stack of relative splittings}

\subsection{Review of the Relative Picard Stack/Functor/Sheaf}\;

Let $B$ be a scheme.
Let $f: Y\to A$ be a proper and flat morphism of finite presentation between $B$-schemes.

Let $B$-$\mathrm{Sch}^{\text{\'et}}$ (resp. $B$-$\mathrm{Sch}^{\mathrm{fppf}}$) be the site of $B$-schemes equipped with the \'etale (resp. fppf) Grothendieck topology. Recall that an algebraic stack over $B$-$\mathrm{Sch}^{\mathrm{fppf}}$ is automatically an algebraic stack over $B$-$\mathrm{Sch}^{\text{\'et}}$ and vice versa, see \cite[\href{https://stacks.math.columbia.edu/tag/076U}{Tag 076U}]{Sta22}. 
In this note, unless stated otherwise, we use algebraic stacks over $B$-$\mathrm{Sch}^{\text{fppf}}$.

Let $\CP^{st}_{Y/A}$ be the relative Picard stack over $A$, i.e., for every $A$-scheme $U$, the groupoid $\CP^{st}_{Y/A}(U)$ consists of all invertible sheaves on $Y_U$; see \cite[\href{https://stacks.math.columbia.edu/tag/0372}{Tag 0372}]{Sta22} for details. 
The stack $\CP^{st}_{Y/A}$ is algebraic; see \cite[\href{https://stacks.math.columbia.edu/tag/0D04}{Tag 0D04}]{Sta22}.

\begin{rmk}
[About the convention on stacks] A stack can be presented as either a special kind of fibered category, or a special kind of lax 2-functor as defined in \cite[Definition 3.10]{Vistoli}. The two presentations are equivalent, see \cite[Proposition 3.11, \S3.1.3]{Vistoli}. In this paper we always present a stack as a lax 2-functor, because it is sometimes more concise. 
Furthermore, note that a lax 2-functor in Vistoli's paper \cite{Vistoli}, which is also called a pseudo-functor in \textit{loc.cit,} is just a functor from a category to a 2-category in Lurie's terminology, see \cite[\href{https://kerodon.net/tag/008K}{Tag 008K}]{Kerodon}. 
In this paper, we abbreviate lax 2-functor to  2-functor.
\end{rmk}

Let $\CP^{fun}_{Y/A}$ be the functor from the category of $A$-schemes to abstract groups that sends an $A$-scheme $U$ to $\frac{Object(\CP^{st}_{Y/A}(U))}{\cong}$, i.e., the group of isomorphism classes of line bundles on $Y_U$; see \cite[\href{https://stacks.math.columbia.edu/tag/0D24}{Tag 0D24}]{Sta22}.

The functor $\CP^{fun}_{Y/A}$ is seldom a sheaf.
Let $\CP^{sh}_{Y/A}$ be the fppf-sheafification of $\CP^{fun}_{Y/A}$. 
This sheaf $\CP^{sh}_{Y/A}$ coincides with the Picard functor as defined in \cite[end of \S2]{Kl05}.

For any scheme $X$, let the abstract group $\CP^{abs}(X)$ be the absolute Picard group of $X$, i.e., the group of isomorphism classes of invertible sheaves on $X$. We have that $\CP^{abs}(Y_U)=\CP^{fun}_{Y/A}(U)$ for all $A$-schemes $U$.
Below we collect some basic properties of the Picard groups, functors, and stacks:

\begin{lem} \label{trivprop}
Let $Y/A$ be flat, proper, and of finite presentation.
If for all $A$-scheme $T$, we have that $\CO_{T}\xrightarrow{\sim}f_{T,*}\CO_{Y_T}$, then: 
\begin{enumerate}
    \item We have the following exact sequence of abstract groups for all $T$:
\[0\to \CP^{abs}(T)\to \CP^{abs}(Y_T) \to \CP^{sh}(Y/A)(T).\]
\item  The fppf sheaf $\CP^{sh}_{Y/A}$ is an algebraic space, quasi-separated and locally of finite presentation over $A$.

\item The stack $\CP^{st}_{Y/A}$ is a $\BG_m$-gerbe over the algebraic space $\CP^{sh}_{Y/A}$.

\item The stack $\CP^{st}_{Y/A}$ is surjective smooth.(in particular locally of finite presentation) over the algebraic space $\CP^{sh}(Y/A)$;

\item The stack $\CP^{st}_{Y/A}$ is locally of finite presentation over $A$.

\item If, furthermore, we have that $f:Y\to A$ is of relative dimension $\le 1$, then both $\CP^{st}_{Y/A}$
and $\CP^{sh}_{Y/A}$ are smooth over $A$.
\end{enumerate}
\end{lem}
\begin{proof}
We just give the reference for the proofs in the Stacks Project:
(1) is proved in \cite[\href{https://stacks.math.columbia.edu/tag/0D27}{Tag 0D27}]{Sta22}; (2) is proved in \cite[\href{https://stacks.math.columbia.edu/tag/0D2C}{Tag 0D2C}, \href{https://stacks.math.columbia.edu/tag/0DNI}{Tag 0DNI}]{Sta22};
(3) is proved in \cite[\href{https://stacks.math.columbia.edu/tag/0DME}{Tag 0DME}]{Sta22}; (4) follows from (3) and \cite[\href{https://stacks.math.columbia.edu/tag/0DN8}{Tag 0DN8}, \href{https://stacks.math.columbia.edu/tag/0DNP}{Tag 0DNP}]{Sta22}; 
(5) follows from (2), (4) and \cite[\href{https://stacks.math.columbia.edu/tag/06Q3}{Tag 06Q3}]{Sta22};
(6) is proved in \cite[\href{https://stacks.math.columbia.edu/tag/0DPJ}{Tag 0DPJ}, \href{https://stacks.math.columbia.edu/tag/0DPK}{Tag 0DPK}]{Sta22}.
\end{proof}

\subsection{The stack of relative splittings}\;

Let $k$ be an algebraically closed field of characteristic $p>0$.
Let $C$ be an integral smooth projective curve of genus $g>1$ over $k$.
Let $Fr: C\to C^{(1)}$ be the relative Frobenius morphism over $k$.
By \cite[Theorem 2.2.3]{BMR}, the sheaf of crystalline differential operators $D_C$ over $C$ gives rise to an Azumaya algebra $\CD$ over the cotagent bundle $T^*C^{(1)}$ of $C^{(1)}$, with an identity of $\CO_{C^{(1)}}$-algebras:
\begin{equation}
    \label{dcd}
    Fr_*D_C=\pi_*\CD,
\end{equation}
where $\pi:T^*C^{(1)}\to C^{(1)}$ is the natural projection. 

Let $A(C,r)=\bigoplus_{i=1}^rH^0(C,\omega_C^{\otimes i})$ be the degree $r$ Hitchin base for $C$. In this paper, we will use the Hitchin base $A(C^{(1)},r)$ more often. We will drop the decoration $r$ when it is clear in the context.

Let $Y\to C^{(1)}\times_k A(C^{(1)})$ be the universal spectral degree $r$ curve over $C^{(1)}$.
Let $a:Y\to T^*C^{(1)}$ be the composition of morphisms of $k$-schemes:
\[a: Y\hookrightarrow T^*C^{(1)}\times_k A(C^{(1)}) \to T^*C^{(1)},\]
where the first arrow is the natural inclusion, and the second arrow is the projection morphism.
The pullback $\CD_{Y}:= a^*\CD$ is an Azumaya algebra over $Y$. 
Given a closed point $b$ in $A(C^{(1)})$, the base change $Y_b:=Y\times_{A(C^{(1)})} b$ is the corresponding spectral curve, and the restriction $\CD_{Y}|_{Y_b}$ coincides with the restriction $\CD|_{Y_b}$ of $\CD$ on $T^*C^{(1)}$ to the subscheme $Y_b\subset T^*C^{(1)}$.

A stack that plays a central role in the Non--Abelian Hodge type theorems in characteristic $p$ is the stack of relative splittings of $\CD_{Y}$:

\begin{defn}[The stack of relative splittings]\label{defS}
We define the stack $\CS$ of relative splittings (of $\CD$ over $Y/A$) to be the 2-functor that sends an $A(C^{(1)})$-scheme $U$ to the following groupoid $\CS(U)$:
\begin{enumerate}
    \item Objects of $\CS(U)$ consists of pairs $(E,\phi)$, where 
    \begin{enumerate}
        \item $E$ is a vector bundle on $Y\times_{A(C^{(1)})} U=: Y_U$;
        \item $\phi: End_{\mathcal{O}_{Y_U}}(E)\xrightarrow{\sim} \CD_{Y}|_{Y_U}$ is an isomorphism of $\mathcal{O}_{Y_U}$-algebras.
    \end{enumerate}
    \item Given two objects $(E,\phi)$ and $(E',\phi')$, a morphism between them is an isomorphism of vector bundles $E\xrightarrow{\sim} E'$ that is compatible in the evident way with $\phi$ and $\phi'$.
\end{enumerate}
\end{defn}

Below we often simply write $E$ for a splitting $(E,\phi)$.

\begin{prop}
\label{pstor}
The 2-functor $\CS$ is a smooth algebraic stack over $A(C^{(1)})$.
Furthermore, it is a torsor under the Picard stack $\CP^{st}_{Y/A}$, i.e., there is a morphism of algebraic stacks
\[act: \CP^{st}_{Y/A}\times_{A(C^{(1)})} \CS\to \CS,\]
such that for every $A(C^{(1)})$-scheme $U$ and every object $L\in \CP^{st}_{Y/A}(U)$, the induced morphism $act(L,-)$ defines an auto-equivalence of the groupoid $\CS(U)$.
\end{prop}
\begin{proof}
This is proved in \cite[Theorem 3.24, Lemma 3.26]{Groch}, with the caveat in Remark \ref{gerbenot} below. For general reductive groups, this is proved in \cite[Lemma 3.7, Theorem 3.8]{ChZh}. We indicate the idea of the proof: since vector bundles satisfy fppf descent, it is routine to check that $\CS$ is indeed a stack over $A(C^{(1)})$. Since any two splittings of an Azumaya algebra on a scheme differ by tensoring by a line bundle, and since such splittings do exist, we see that $\CS$ is a torsor under $\CP^{st}_{Y/A}$. The algebraicity and smoothness of $\CS$ then follow from those for $\CP^{st}_{Y/A}$, because they can be checked locally in the \'etale topology. 
\end{proof}

\begin{rmk}
\label{gerbenot}
Note that in \cite[Definition 3.22]{Groch}, the morphisms in the groupoid $\CS(U)$ are defined to be  pairs $(L,\gamma)$ where $L$ is a line bundle on $Y_U$ and $\gamma: E\xrightarrow{\sim} F\otimes L$. We do not adopt this definition because  under this definition, if $\CS(U)$ is nonempty, then all objects in it are isomorphic to each other.
In the definition we adopt, the groupoid $\CS(U)$ is equivalent to the groupoid $\CP^{st}_{Y/A}(U)$.
\end{rmk}

In \cite[Theorem 2.2.3]{BMR}, it is proved that the Azumaya algebra $\CD$ on $T^*C^{(1)}$ is not trivial. Using a similar argument, we show that 

\begin{lem}\label{notriv}
The $\CP^{st}_{Y/A}$-torsor $\CS$ is not the trivial torsor.
\end{lem}
\begin{proof}
By contradiction, assume that $\CS$ is the trivial torsor.
Then there is a section $s: A(C^{(1)})\to \CS$.
By the definition of the stack $\CS$, the section $s$ defines a splitting of the Azumaya algebra $\CD_{Y}$ on the universal spectral curve $Y$.
Thus there is a vector bundle $E$ of rank $p$ on $Y$ with $\CD_{Y}\cong End_{\CO_{Y}}(E)$.
Zariski locally over $Y$, the $\CO_{Y}$-algebra $\CD_{Y}\cong End_{\CO_{Y}}(E)$ is isomorphic to the algebra $End_{\CO_{Y}}(\CO_{Y}^{\oplus p})$, which has non-zero zero divisors.

Claim: Zariski locally, the $\CO_{Y}$-algebra $\CD_{Y}$ does not have non-zero zero divisors.
Clearly, proving the claim proves the lemma. 

The following proof  of the lemma uses \cite[Theorem 2.2.3]{BMR} to the effect that, Zariski-locally over $T^*C^{(1)}$, the Azumaya algebra $\CD$ does not have non-zero zero divisors. 

We can replace $Y$ with an affine scheme $Spec(R')$ and $T^*C^{(1)}$ with an affine scheme $Spec(R)$. We are reduced to showing that the algebra $\CD\otimes_{R} R'$ does not have non-zero zero divisors. 
By Lemma \ref{affb} below, we have that $Y$ is an affine bundle over $T^*C^{(1)}$. By shrinking $Spec(R)$ if necessary, we can assume that the affine bundle is a vector bundle and that $R'$ has the form $R[t_1,...,t_n]$.
We are thus reduced to showing that $\CD\otimes_R R[t_1,...,t_n]$ has no non-trivial zero divisors. This clearly follows from the case $n=0$, which is proved in \cite[Theorem 2.2.3]{BMR}.
\end{proof}

\begin{lem}
\label{affb}
Let $C$ be a connected smooth proper curve over a field $k.$
Let $Y_C$ be the degree $r$ universal spectral curve over $C\times_k A(C,r).$
Then $Y_C$ is an affine space bundle over the cotangent bundle $T^*C.$
\end{lem}
\begin{proof}
Let $S:=\BP(\omega_C\oplus\CO_C)\cong \BP(\CO_C \oplus TC )$ be the projectivization of the cotangent bundle $T^*C.$
Let $\pi: S\to C$ be the projection morphism.
Let $\CO_S(1)$ be the relative  hyperplane bundle with $\pi_*\CO_S(1)=\omega_C\oplus \CO_C.$
Since $\omega_C\oplus\CO_C$ is base point free, we have that $\CO_S(1),$ thus $\CO_S(n),$ is also base point free.
Let $Y''/\BP(\omega_C\oplus\CO_C)$ be the universal family for the linear system $|\CO_S(n)|.$
By the base point freeness of $\CO_S(n),$ we have that $Y''/\BP(\omega_C\oplus\CO_C)$ is a projective space bundle.
Therefore, the restriction $Y'/T^*C$ of $Y''$ is a projective space bundle over $T^*C.$

We now describe the universal spectral curve $Y_C/T^*C$ inside $Y'/T^*C.$
Let $\Ft\in H^0(S,\CO_S(1))=H^0(C,\omega_C\oplus \CO_C)$ be the section
given by $1\in H^0(C,\CO_C).$
We have that $\Ft$ vanishes at the zero section $C\subset T^*C.$
The sections of $\CO_S(n)$ are of the form 
$
    c\Ft^r+(\pi^*a_1)\Ft^{r-1}+...+\pi^* a_r,
$
where $a_i\in H^0(C,\omega_C^{\otimes i}), i=1,...,r,$ and $c\in H^0(C,\CO_C)\cong k.$
We have that $Y_C\subset Y'$ is given by the condition that $c\ne 0.$
Therefore, the affine space bundle structure on $Y_C/T^*C$ follows from the projective space bundle structure on $Y'/T^*C.$
\end{proof}

\section{The stack of very good splittings}

\subsection{The stack $\CP^{st,o}(Y/A)$}\;

The canonical reference for the definition of the identity and torsion components of the relative Picard functor $\CP^{sh,o}(Y/A)$ and $\CP^{sh,\tau}(Y/A)$ is \cite[XIII, \S4]{sga6}. 
Namely, given an $A$-scheme $U$, the group $\CP^{sh,o}(Y/A)$ (resp. $\CP^{sh,\tau}(Y/A)$ consists of the isomorphism classes of line bundles $L$ on $Y_{U}$ such that for each geometric point $u$ of $U$, $L|_u$ (resp. some nonzero integer power of $L|_u$) is algebraically equivalent to the trivial line bundle on $Y_u$.

According to \cite[\S9, Cor. 14]{BLR}, given $Y/k$ a proper curve over an algebraically closed field $k$ with $r$ irreducible components $Y_1,...,Y_r$, the map:
\begin{equation}
    \label{inj}
    \CP^{abs}(Y)/\CP^{abs,o}(Y)\to \BZ^r,\;\;\; L\mapsto (\mathrm{deg}_{Y_1}(L),...,\mathrm{deg}_{Y_r}(L)),
\end{equation}
is injective and has finite index. The degree of a line bundle is defined as the degree of any of the associated Cartier divisors, which is in turn defined in \cite[p.237]{BLR}.

By the injection (\ref{inj}) above, given a proper curve $Y$ over an algebraically closed field $k$, we have the canonical natural equivalence of functors
$\CP^{sh,o}_{Y/k}\cong \CP^{sh,\tau}_{Y/k}.$
In particular, in our case, where $Y/A$ is a proper morphism between schemes so that all the geometric fibers are curves, we also have a natural canonical equivalence of functors:
\begin{equation}
    \label{oandtau}\CP^{sh,o}_{Y/A}\xrightarrow{\sim} \CP^{sh, \tau}_{Y/A}.
\end{equation}

\begin{defn}
\label{defpiost}
We define $\CP^{st,o}_{Y/A}$ to be the fiber product of algebraic stacks 
\[\CP^{st,o}_{Y/A}:= \CP^{st}_{Y/A}\times_{\CP^{sh}_{Y/A}} \CP^{sh,o}_{Y/A}.\]
\end{defn}

We record some useful properties of the algebraic spaces and stacks above:

\begin{prop}\label{picost}
Let $Y/A=Y/A(C^{(1)})$ be the universal spectral curve. Then
\begin{enumerate}
\item $\CP^{sh,o}_{Y/A(C^{(1)})}$ is representable by a smooth quasi-projective group scheme over $A(C^{(1)})$.
\item Set $d^o:=r(r-1)(1-g)$. There is a canonical open immersion $\CP^{sh,o}_{Y/A(C^{(1)})}\hookrightarrow \CM_{Dol}^s(C^{(1)},r,d^o)$ into the moduli space of stable Higgs bundles of rank $r$ and degree $d^o$ over $C^{(1)}$.
\item For each $A$-scheme $U$, $\CP^{sh,o}_{Y/A}(U)$ consists of isomorphism classes of line bundles $L$ on $Y_U$ such that for each geometric point $u\in U$, the isomorphism class of $L|_{Y_u}$ lies in the identity component of $\CP^{abs}(Y_u)$.
\item $\CP^{sh, o}_{Y/A}$ is an open algebraic subspace of $\CP^{sh}_{Y/A}$.
\item The stack $\CP^{st,o}_{Y/A}$ is a $\BG_m$-gerbe over $\CP^{sh}_{Y/A}$.
    \item $\CP^{st,o}_{Y/A}$ is an open substack of $\CP^{st}_{Y/A}$.
    \item $\CP^{st,o}_{Y/A}$ is smooth over $A$.
    \item For each $A$-scheme $U$, $\CP^{st,o}_{Y/A}(U)$ is the groupoid of line bundles $L$ on $Y_U$ such that for each geometric point $u\in U$, $L|_{Y_u}$ lies in the identity component of $\CP^{abs}(Y_u)$.
\end{enumerate}
\end{prop}
\begin{proof}
Items (1) and (2) are proved in \cite[p.716, \S5]{CL}. 
Items (3) and (8) follows from \cite[Cor. 4.18.3, Prop. 5.10]{Kl05}, which states that for a complete scheme $Y$ over a field $k$, the sheaf $\CP_{Y/k}^{sh}$ is representable by a group scheme, and the identity component of this group scheme represents the sheaf $\CP_{Y/k}^{sh, o}$. 
Item (4) follows from the isomorphism (\ref{oandtau}) and \cite[XIII, Theorem 4.7]{sga6}, which states that for any proper morphism of finite presentation $Y/A$ over a quasi-compact base $S$ (and here $S=Spec(k)$), the canonical inclusion $\CP^{sh,\tau}_{Y/A}\hookrightarrow \CP^{sh}_{Y/A}$ is representable by a quasi-compact open immersion.
Item (5) follows from Lemma \ref{trivprop}.(3).
Item (6) follows from (4) above and \cite[\href{https://stacks.math.columbia.edu/tag/0501}{Tag 0501}]{Sta22}.
Item (7) follows from (6) and Lemma \ref{trivprop}.(4).
\end{proof}

\subsection{Ranks and Degrees}
\label{rkdeg}
Before introducing the stack $\CS^o$, let us recall the definition for ranks and degrees of torsion free sheaves on spectral curves.
This material is standard, e.g, see \cite[\S2.4]{dCSL}.

Let us fix a spectral curve $\Gamma/C$. We can write $\Gamma=\sum_k m_k \Gamma_k$, where each $\Gamma_k$ is an integral curve and $\Gamma_k\ne \Gamma_{k'}$ if $k\ne k'$.
Let $\eta_k$ be the generic point of each $\Gamma_k$.
We say that a coherent $\CO_{\Gamma}$-module $E$ is torsion free if, for all the natural inclusions of the generic points $j_k: \eta_k\to \Gamma$, the unit morphism $E\to \prod_k j_{k,*}j_k^*E$ is injective. 

Given a torsion free sheaf $E$ on $\Gamma$, we say that $E$ has rank $rk_{\Gamma}(E)=r$ if for each $k$, we have that $l_{\CO_{\eta_k}}(E_{\eta_k})=r m_k$.
When such a rank is well defined, we define the degree to be:
\begin{equation}\label{defdeg}
\mathrm{deg}_{\Gamma}(E):=\chi(E)-rk_{\Gamma}(E)\chi(\CO_{\Gamma}).
\end{equation}

Recall that given a Higgs bundle $E$ with spectral curve $\Gamma$, there is a unique coherent $\CO_{\Gamma}$-module $\CB(E)$ with an identity $\pi_{C,*}\CB(E)=E$ of $\pi_{C,*}\CO_{T^*C}$-modules, where $\pi_C:T^*C\to C$ is the projection. The coherent $\CO_{\Gamma}$-module:
\begin{equation}
    \label{bnr}
    \CB(E)
\end{equation}
is called the BNR sheaf of $E$. By \cite[Proposition 2.1]{Schaub}, BNR sheaves are all torsion free of rank 1 according to our definition. We define the degrees of BNR sheaves using (\ref{defdeg}) above.

\begin{rmk}[On the definition of ranks and degrees as in \cite{Groch}]
In \cite[\S3.6]{Groch} the author defined rank and degree only for sheaves on $\Gamma$ that admit a finite length resolution by vector bundles. However, this is not sufficient in \cite{Groch} as well as in this paper.
In particular, in \cite[Lemma 3.25]{Groch}, the author states that BNR sheaves have finite length resolution by vector bundles. This is not true in general. For a counter-example, consider a spectral curve with  a node at a point $x$. The maximal ideal sheaf $\mathfrak{m}_x$ is a torsion free rank 1 sheaf, so a BNR sheaf of some Higgs bundle. However, its projective dimension is infinite, thus it cannot have a finite length resolution by vector bundles. 
In general,  we have the following lemma:

\begin{lem}
A BNR sheaf either is a vector bundle, or it has infinite projective dimension.
\end{lem}
\begin{proof}
As stated correctly in the proof of \textit{loc.cit,} a BNR sheaf $\CF$ on a spectral curve $\Gamma$ is a maximal Cohen-Macaulay sheaf, i.e., for each point $x\in X$, we have $\mathrm{depth}_{\CO_{X,x}}(\CF_x)=\dim(\CO_{X,x})=\mathrm{depth}(\CO_{X,x})$. The Auslander-Buchsbaum Theorem \cite[\href{https://stacks.math.columbia.edu/tag/090U}{Tag 090U}]{Sta22} states that if $\CF_x$ is an $\CO_{X,x}$-module of finite projective dimension, then $\mathrm{depth}(\CO_{X,x})=\mathrm{pd}_{\CO_{X,x}}(\CF_x)+\mathrm{depth}_{\CO_{X,x}}(\CF_x)$, thus $\mathrm{pd}_{\CO_{X,x}}(\CF_x)=0$, i.e., $\CF$ is already a vector bundle. 
\end{proof}

We note that this incorrect definition of rank and degree has no consequence on the result of \cite{Groch}. In fact:
\begin{enumerate}
    \item 
    Since the author also defines degrees of sheaves using the Riemann-Roch formula, as long as the torsion free sheaf has a well-defined integer rank, the degree of the sheaf agrees with the one in (\ref{defdeg});
    \item The proof of \cite[Lemma 3.37]{Groch} only involves Riemann-Roch calculation of degrees, so the proof and the conclusion are correct;
    \item \cite[Corollary 3.38]{Groch} is an immediate corollary of \textit{loc.cit,} so it is also correct;
    \item The contents from Lemma 3.39 to Corollary 3.45 in \cite{Groch} are only concerned with the behavior of degrees of some line bundles or vector bundles, and they are thus correct.
\end{enumerate}
\end{rmk}

\subsection{Introducing the stack $\CS^{o}$ of very good splittings}\;

Recall that the stack of relative splittings $\CS$ is defined in Definition \ref{defS}.
\begin{defn}[Very Good Splittings]
\label{defvgs}
We define $\CS^o$ to be the 2-functor such that for each $A$-scheme $U$, the groupoid $\CS^o(U)$ is the full subgroupoid of $\CS(U)$ consisting of splittings $E$ satisfying the following property:
For every geometric point $u$ of $U$, write $Y_u=\sum_i m_i\Gamma_i$, where each $\Gamma_i$ is integral and $\Gamma_i\ne \Gamma_{i'}$ for $i\ne i'$. Each $\Gamma_i$ is a spectral curve of degree $r_i$ over $C^{(1)}$. We have that $\sum_i m_ir_i=r$. For each $i$, we require that
\begin{equation}
    \label{degcond}
    \mathrm{deg}_{\Gamma_i}(E_u|_{\Gamma_i})=(1-p)\chi(\CO_{C^{(1)}}) r_i.
\end{equation}

We call such a splitting $E$ in $\CS^o(U)$ a very good splitting on $Y_U$.
\end{defn}

Let us explain why it is natural to come up with this condition (\ref{degcond}) on the degrees. Let $a$ be a geometric point of $A(C^{(1)})$. Let $Y^p_a:= Y_a\times_{C^{(1)}} C$, which lies inside $T^{1,*}_C:=Tot(\omega_{C}^{\otimes p})=T^*_{C^{(1)}}\times_{C^{(1)}} C$. 
Let $E$ be a splitting of $\CD|_{Y_a}$.
Let $W: T^{1,*}_C\to T^*C^{(1)}$ be the natural projection.
By \cite[Lemma 2.1.1]{BMR}, there is a natural inclusion of $\CO_{T^*C^{(1)}}$-algebras $W_*\CO_{T^{1,*}_C}\hookrightarrow\CD$.
Let $i: Y_a\hookrightarrow T^*C^{(1)}$ be the closed immersion. 
Since $W$ is affine, by cohomology and base change \cite[\href{https://stacks.math.columbia.edu/tag/02KG}{Tag 02KG}]{Sta22}, we have an isomorphism of $\CO_{T^*C^{(1)}}$-algebras $W_*i^*\CO_{T^{1,*}_C}\cong i^*W_*\CO_{T^{1,*}_C}$.
We thus have the following morphisms of $\CO_{T^*C^{(1)}}$-algebras:
\begin{equation}
    \label{wbc}
    W_* \CO_{Y_a^p}=W_*i^*\CO_{T^{1,*}_C}\cong i^*W_*\CO_{T^{1,*}_C}\hookrightarrow i^*\CD=\CD|_{Y_a}.
\end{equation}
We can thus view the $\CD|_{Y_a}$-module $E$ as a $W_*\CO_{Y_a^p}$-module.
Since $W$ is affine, by \cite[Proposition 12.5]{GW}, there is an $\CO_{Y_a^p}$-module $L_E$ on $T^{1,*}C$ with $W_*L_E=E$. 
By \cite[Proposition 2.3]{OV}, we have that $L_E$ is indeed a line bundle on $Y_a^p$.
\begin{defn}[$L_E$]
\label{defle}
Given a splitting $E$ on a spectral curve $Y_a$, we define $L_E$ to be the unique line bundle on $Y_a^p$ with an equality of $W_*\CO_{Y_a^p}$-modules $W_*L_E=E$.
\end{defn}

\begin{lem}\label{vglb}
A splitting $E$ is very good if and only if $L_E\in \CP^{st,o}_{Y^p/A}(U).$
\end{lem}
\begin{proof}
We can write $Y_u^p=\sum_i m_i\Gamma_i^p$, where each $\Gamma_i^p:=\Gamma_i\times_{C^{(1)}}C$ is integral, and any two of them are distinct. Since $W: \Gamma_i^p\to \Gamma$ is an affine morphism, we have the equality $\chi(L_E)=\chi(E)$. Furthermore, each $\Gamma_i^p$ is a degree $r_i$ spectral curve over $C$ for the bundle $\omega_C^{\otimes p}$. By Riemann-Roch, we thus have that 
\[\mathrm{deg}_{\Gamma_i^p}(L_E)+\chi(\CO_{\Gamma_i^p})= \mathrm{deg}_{\Gamma_i}(E)+p\chi(\CO_{\Gamma_i}).\]
Since a degree $r_i$ spectral curve $Y_D$ over $C$ for a line bundle $D$ satisfies the identity $\chi(\CO_{Y_D})= r_i\chi(\CO_C)-\frac{r_i(r_i-1)}{2}deg(D),$ we have that 
\begin{equation*}
     \mathrm{deg}_{\Gamma_i^p}(L_E)+r_i\chi(\CO_C)-\frac{r_i(r_i-1)}{2}\mathrm{deg}(\omega_C^{\otimes p})
    =  \mathrm{deg}_{\Gamma_i}(E)+p\Big(r_i\chi(\CO_{C^{(1)}})-\frac{r_i(r_i-1)}{2}\mathrm{deg}(\omega_{C^{(1)}})\Big).
\end{equation*}
In particular, we have that $\mathrm{deg}_{\Gamma_i}(E)=(1-p)r_i \chi(\CO_{C^{(1)}})$ if and only if $\mathrm{deg}_{\Gamma_i^p}(L_E)=0$. By \cite[\S9, Proposition 5]{BLR}, we have that $\mathrm{deg}_{m_i\Gamma_i^p}(L_E)=m_i \mathrm{deg}_{\Gamma_i^p}(L_E)$. Therefore, we have that $\mathrm{deg}_{\Gamma_i}(E)=(1-p)r_i \chi(\CO_{C^{(1)}})$ if and only if $L_E\in \CP^{st,o}_{Y^p/A}(U).$
\end{proof}

\begin{cor}
There is a morphism of algebraic stacks $\CS\to \CP^{st}_{Y^p/A}$, inducing an isomorphism of algebraic stacks:
\begin{equation}
\label{soopen}
    \CS^o\cong \CS\times_{\CP^{st}_{Y^p/A}} \CP^{st,o}_{Y^p/A}.
\end{equation}
\end{cor}
\begin{proof}
The morphism $\CS\to \CP^{st}_{Y^p/A}$ is given by sending a splitting $E$ on $Y_U$ to the line bundle $L_E$ on $Y_U^p$.
By the inclusion (\ref{wbc}), we see that a morphism of splittings $E\to E'$ defines a morphism of line bundles $L_{E}\to L_{E'}$.
Therefore we indeed get a morphism of stacks. The isomorphism (\ref{soopen}) then follows from Lemma \ref{vglb}.
\end{proof}

\begin{lem}
\label{poso}
The stack $\CS^o$ is a torsor under the Picard stack $\CP^{st,o}_{Y/A}.$
\end{lem}
\begin{proof}
Given Proposition \ref{pstor}, we know that for an $A(C^{(1)})$-scheme $U$, two splittings in $\CS^o(U)$ differ by a unique up to an isomorphism line bundle $R_U$ on $Y_U:=Y\times_{A(C^{(1)})}U.$ By the degree restrictions on $\CS^o$, we see that $R_U$ has to lie in $\CP^{st,o}_{Y/A}(U)$.
\end{proof}

\begin{prop}\label{vgopen}
$\CS^o$ is an open algebraic substack of $\CS$ over $Sch_{A(C^{(1)})}^{fppf}$.
\end{prop}
\begin{proof}
The Proposition immediately follows from the isomorphism (\ref{soopen}) and Theorem \ref{picost}.(1).
\end{proof}

\begin{cor}\label{vgsm}
The stack $\CS^o$ is smooth over $A$.
\end{cor}
\begin{proof}
The smoothness of $\CS/A$ is proved in \cite[Lemma 3.25]{Groch}. 
Combined with Proposition \ref{vgopen}, we obtain the Corollary, since open immersion is smooth and composition of smooth morphisms is also smooth.
\end{proof}

\begin{rmk}
Proposition \ref{vgopen} implies that the morphism $\CS^o\to \CS$ is universally open. Below we give a geometric proof of this fact:
\begin{lem}\label{open1}
The morphism of stacks $\CS^o\hookrightarrow\CS$ is universally open. 
\end{lem}
\begin{proof}
Let $U$ be an $\CS$-scheme and let $U^o:=U\times_{\CS}\CS^o$ be the base change.
We need to show that for all $U$-scheme $T$ with $T^o:=T\times_U U^o$, we have that $T^o\to T$ is open. 
All the assumptions in Lemma \ref{monost} below are satisfied, thus $T^o\hookrightarrow T$ is a finite type universally injective morphism of schemes. 
By noetherian approximation, we can assume that $T$ and $T^o$ are noetherian. 
Therefore, the morphism $T^o\hookrightarrow T$ is of finite presentation \cite[\href{https://stacks.math.columbia.edu/tag/01TX}{Tag 01TX}]{Sta22}.
By Chevalley's Theorem \cite[\href{https://stacks.math.columbia.edu/tag/054K}{Tag 054K}]{Sta22}, we have that $|T^o|$ is constructible inside $|T|$. 
By \cite[\href{https://stacks.math.columbia.edu/tag/0542}{Tag 0542}, \href{https://stacks.math.columbia.edu/tag/054F}{Tag 054F}]{Sta22}, 
to show that $|T^o|\subset|T|$ is open, it suffices to show that given a DVR $R$ with generic point $\eta$ and special point $x$, and an $R$-point of $T$, if the image of $\eta$ lies in $T^o$ then the image of $x$ also lies in $T^o$. Considering the definition of $\CS$ and $\CS^o$, we are reduced to showing the following:
given a splitting $E_{\eta}$ over $Y_{\eta}$ that specializes to a very good splitting $E_x$ over $Y_x$, then $E_{\eta}$ is also very good. 

Let $H_R$ be a relatively ample line bundle on $C_R^{(1)}/\Spec{R}$. Let $H_x$ be of degree $d$ on $C^{(1)}_x$. Let $\pi_R: Y_R\to C_R$ be the projection. Let $M_R:=\pi_R^* H_R$. We have that $M_R$ is relatively ample on $Y_R/R$.
Write $Y_x=\sum_i \Gamma_i$ where $\Gamma_i$'s are pairwise distinct irreducible spacetral curves with degree $r_i$.
We have that the pullback $\mathrm{deg}_{\Gamma_i}(M_x)=dr_i$.
Let $L_R:=\mathrm{det}(E_R)^{\otimes d}\otimes M_R^{(p-1)(1-g)}$.
Since $E_x$ is very good, we have that $\mathrm{deg}_{\Gamma_i}(E_x)=r_i(1-p)(1-g)$.
Therefore we have that $L_x$ is in $\CP^{abs, o}(Y_x)$. Apply \cite[XIII, Theorem 4.6]{sga6}, we have that $L_x^{\otimes n}\otimes M_x$ stays ample for all $n\in\BZ$. By the openness of ampleness, we have that $L_R^{\otimes n}\otimes M_R$ is also ample for all $n$. Apply \textit{loc.cit.} again, we have that $L_{\eta}=\mathrm{det}(E_{\eta})^{\otimes d}\otimes M_{\eta}^{(p-1)(1-g)}$ is in $\CP^{st,o}(Y_{\eta})$. Write $Y_{\eta}=\sum_i\Gamma_{i,\eta}$ where each $\Gamma_{i,\eta}$ is irreducible of degree $r_i'$ over $C$. Since $H_{\eta}$ is also of degree $d$ on $C_{\eta}^{(1)}$, we have that $\mathrm{deg}_{\Gamma_{i,\eta}}(M_{\eta})=dr_i'$. Combined with the fact that $\mathrm{deg}_{\Gamma_{i,\eta}}(L_{\eta})=0$, we see that $E_{\eta}$ is a very good splitting. 
\end{proof}

The following lemma is used in the proofs of Lemma \ref{open1} and Theorem \ref{vg is s}. It is well-known, but we cannot find a suitable reference.
\begin{lem}\label{monost}
Let $A$ be a scheme.
Let $N$ and $M$ be two algebraic stacks over $Sch^{fppf}_A$ that are (resp. locally) of finite type over $A$. 

Suppose that, when viewed as categories, there is a fully faithful functor $f: N\to M$, then $f$, when viewed as a morphism between algebraic stacks, is representable by a monomorphism of schemes, (resp. locally) of finite type, i.e., given any $M$-scheme $T$ with $T_N:T\times_M N$, we have that $T_N\to T$ is a monomorphism of schemes, (resp. locally) of finite type.

In particular, the morphism $f$ is representable by a universally injective morphism of schemes, i.e., $T_N\to T$ is a universally injective morphism of schemes.
\end{lem}
\begin{proof}
Let $T$ be an $M$-scheme. Let $T_N:=T\times_M N$ be the fiber product.
By \cite[\href{https://stacks.math.columbia.edu/tag/04ZZ}{Tag 04ZZ}]{Sta22}, fully faithful functors between algebraic stacks are representable by algebraic spaces and are monomorphisms of algebraic stacks.  
Since both $N$ and $M$ are (resp. locally) of finite type over $A$, we have that $f$ is (resp. locally) of finite type \cite[\href{https://stacks.math.columbia.edu/tag/06U9}{Tag 06U9}]{Sta22}.
Since being a monomoprhism and being (resp. locally) of finite type are stable under base changes \cite[\href{https://stacks.math.columbia.edu/tag/04ZX}{Tag 04ZX}, \href{https://stacks.math.columbia.edu/tag/06FU}{Tag 06FU}]{Sta22}, we have that $T_N\to T$ is a monomorphism of algebraic algebraic spaces, locally of finite type.
By \cite[\href{https://stacks.math.columbia.edu/tag/0B89}{Tag 0B89}]{Sta22}, we have that $T_N\to T$ is a morphism of schemes, locally of finite type, which is also a monomorphism of algebraic spaces. By the beginning of \cite[\href{https://stacks.math.columbia.edu/tag/042K}{Tag 042K}]{Sta22}, we have that $T_N\to T$ is indeed a monomorphism of schemes. By \cite[Proposition 17.2.6]{EGA4}, we have that $T_N\to T$ is radicial. Finally, by \cite[\href{https://stacks.math.columbia.edu/tag/01S4}{Tag 01S4}]{Sta22}, we have that radicial is equivalent to universally injective, thus we have that $T_N\to T$ is universally injective.
\end{proof}
\end{rmk}

\begin{lem}\label{gminv}
The image of $\CS^o$ inside $A$ is invariant under the $\BG_m$-action on $A$.
\end{lem}
\begin{proof}
We imitate the proof of \cite[Lemma 3.43]{Groch}. 
Let $a\in A(k)$ be a closed point of $A$ corresponding to a spectral curve $Y_a/C^{(1)}$.
The morphism $Y\to A$ is $\BG_m$-equivariant for the natural $\BG_m$-action on $A$ and the $\BG_m$-action on $Y$ induced by homotheties on $T^*C^{(1)}$.
Under the $\BG_m$-action on $A$, the orbit of $a$ has the form $t\cdot a$, which corresponds to a canonically trivializable $\BG_m$-family of spectral curves $Y_{t\cdot a}$. The cohomology class of $\CD|_{Y_{\BG_m\cdot a}}$ in $H^2(Y_{\BG_m\cdot a},\BG_m)$ is not necessarily trivial. 

Let $E$ be a splitting of $\CD$ on $Y_a$. Let $Y_a=\sum_{i=1}^n m_i \Gamma_i$, where $\Gamma_i$'s are pairwise distinct integral curves. By \cite[p.241, Cor.8; p.251, Cor 14]{BLR} the line bundles on $m_i\Gamma_i$ has degrees $m_id_i\BZ$ for some integer $d_i$.
By keeping in mind that the rank of $E$ is $p$, we see that all the splittings on $Y_a$ have multidegrees giving the subset:
\[\CR_a:= (\mathrm{deg}_{m_i\Gamma_i}(E)+m_id_ip\BZ)_{i=1}^n\subset \BZ^n.\]
We are done if we can show that $\CR_a=\CR_{t\cdot a}$ for all $t\in \BG_m(k)$, since this entails that $\CR_{t\cdot a}$ also contains the element $(\mathrm{deg}_{m_i\Gamma_i}(E))_i\in \BZ^n$, and a splitting of this multidegree is very good.
Since $\CR_a$ and $\CR_{t\cdot a}$ consist of direct products of arithmetic progressions, we have that $\CR_a=\CR_{t\cdot a}$ if and only if $\CR_a\cap \CR_{t\cdot a}\ne \emptyset$. 

Let $E_a$ and $E_{t\cdot a}$ be two splittings on $Y_a$ and $Y_{t\cdot a}$ respectively. They have multidegree $\Vec{w}_a\in \CR_{a}$ and $\Vec{w}_{t\cdot a}\in \CR_{t\cdot a}$. By the smoothness of $\CS/A$, there exists \'etale neighborhoods $U_a$ and $U_{t\cdot a}$ of $a$ and $t\cdot a$ inside the orbit $\BG_m\cdot a$, where we can extend $E_a$ and $E_{t\cdot a}$ to splittings $E_{U_a}$ and $E_{U_{t\cdot a}}$ over $U_a$ and $U_{t\cdot a}$ respectively. 
Now pick a closed point $b$ in the image of the fiber product $U_a\times_{\BG_m\cdot a} U_{t\cdot a}$ in $\BG_m\cdot a$.
We have that $E_{U_a}|_{Y_b}$ is a splitting on $Y_b$ with multidegree $\Vec{w}_b$.
We thus have that $\Vec{w}_b=\Vec{w}_a$, hence $\CR_a=\CR_b$.
Similarly, let $\Vec{w}_{b}'$ be the degree of $E_{U_{t\cdot a}}|_{Y_b}$. We have that $\Vec{w}_b'=\Vec{w}_{t\cdot a}$, thus we have that $\CR_{t\cdot a}=\CR_b=\CR_a$. 
\end{proof}

\begin{lem}\label{vgono}
There is a very good splitting on the nilpotent curve $rC^{(1)}$. 
\end{lem}
\begin{proof}
By \cite[Remark 2.2.5]{BMR}, we have that $Fr_*\CO_C$ is a splitting of $\CD|_{C^{(1)}}$.
By Lemma \ref{vglb}, $Fr_*\CO_C$ very good because (we use the notation $L$ as in Definition \ref{defle}) $L_{Fr_*\CO_C}=\CO_C\in \CP^{st,o}(C).$
By \cite[Lemma 3.23]{Groch}, there exists a splitting $E$ on $rC^{(1)}$. 
Since two splittings differ by tensoring with a line bundle, we have that $E|_{C^{(1)}}\otimes H\cong Fr_*\CO_C$ for a line bundle $H$ on $C^{(1)}$.

By \cite[Lemma 5.11]{LiuQ}, there exists a line bundle $H'$ on $rC^{(1)}$ with $H'|_{C^{(1)}}=H$. The splitting $E\otimes H'$ of $\CD|_{rC^{(1)}}$ satisfies that \[\mathrm{deg}_{C^{(1)}}(E\otimes H')=\mathrm{deg}_{C^{(1)}}(E|_{C^{(1)}}\otimes H|_{C^{(1)}})=\mathrm{deg}_{C^{(1)}}(E|_{C^{(1)}}\otimes H)=\mathrm{deg}_{C^{(1)}}(Fr_*\CO_C)=(1-p)(1-g).\]
Therefore, the splitting $E\otimes H'$ of $\CD|_{rC^{(1)}}$ is very good.
\end{proof}

Recall that an algebraic stack has an underlying topological space consisting of equivalence classes of field-valued points \cite[\href{https://stacks.math.columbia.edu/tag/04XL}{Tag 04XL}]{Sta22}. 
It is shown in \cite[\href{https://stacks.math.columbia.edu/tag/04XI}{Tag 04XI}]{Sta22} that a morphism of algebraic stacks representable by algebraic spaces is surjective (in the sense of \cite[\href{https://stacks.math.columbia.edu/tag/04XB}{Tag 04XB}]{Sta22}) if and only if it induces a surjective function on the underlying topological spaces.
\begin{prop}
\label{vgsurj}
The natural morphism between $k$-stacks $\CS^o\to A(C^{(1)})$ is surjective.
\end{prop}
\begin{proof}
By Corollary \ref{vgsm}, we have that $\CS^o$ is smooth over $A(C^{(1)})$. Let $U\to A(C^{(1)})$ be a smooth and surjective morphism with $U$ an $A(C^{(1)})$-scheme. Since smooth morphisms between schemes are open \cite[\href{https://stacks.math.columbia.edu/tag/056G}{Tag 056G}]{Sta22}, we have that the set theoretic image of $U$, thus of $\CS^o$, inside $A(C^{(1)})$ is open. Lemma \ref{gminv} and Lemma \ref{vgono} imply that the image is also $\BG_m$-invariant and contains the origin. Therefore we see that the image has to be the whole of $A(C^{(1)})$. 
\end{proof}

\section{Semistability}
\subsection{Reminder on the de Rham-BNR correspondence}\label{secdrbnr}\;

We recall the de Rham-BNR correspondence as in \cite[Prop. 3.15]{Groch}.
We warn the reader that this correspondence is worded differently in \textit{loc.cit.} than in Theorem \ref{drbnr} below because in this paper we employ our definition of rank in \S \ref{rkdeg}.
\begin{thm}[de Rham-BNR]\label{drbnr}
The category of flat connections of rank $r$ on $C$ is equivalent to the category of coherent torsion free sheaves of rank $p$, endowed with a $\CD$-module structure, scheme-theoretically supported on a degree $r$ spectral curve over $C^{(1)}$. 
\end{thm}

The de Rham-BNR correspondence works as follows.
We have the following Cartesian square:
\begin{equation}
\label{4corners}
    \xymatrix{
    T^{1,*} C \ar[r]^-W \ar[d]_-{\pi^p} & T^*C^{(1)}\ar[d]^-{\pi}\\
    C\ar[r]_-{Fr} & C^{(1)}
    }
\end{equation}
and identifications compatible with the evident inclusions of $\CO_{C^{(1)}}$-algebras: 
\[
 Z(Fr_*D_C) =    \pi_*\CO_{T^*C^{(1)}}, \quad
Fr_*D_C = \pi_*\CD.
\]

Given a flat connection $G$ on $C$, $Fr_*G$ is a $Z(Fr_*D_C)=\pi_*\CO_{T^*C^{(1)}}$-module.

It follows that $Fr_*G$ is naturally endowed with a structure of a rank $pr$ Higgs bundle. 

The associated BNR sheaf $\CB(Fr_*G)$ (notation as in (\ref{bnr})) is a rank 1 torsion 
free sheaf with Fitting support a spectral curve $Z$ of degree $pr$ in $T^*C^{(1)}$
which is the $p$-th order thickening of a degree $r$ spectral curve $Y$ in $T^*C^{(1)},$ i.e. $Z=pY;$
moreover, the scheme-theoretic support of $\CB(Fr_*G)$ is indeed inside $Y.$ 

The de Rham-BNR correspondence sends $G$ to   $\CB(Fr_*G)$ on the degree $r$ spectral curve $Y.$

Conversely, let $\CG$ be a rank $p$ torsion free sheaf with a $\CD$-module structure on a degree $r$ spectral curve $Y$.

There is a natural $Fr_*D_C=\pi_*\CD$-module structure on the rank $pr$ vector bundle $\pi_*\CG$.

The inverse to the de Rham-BNR correspondence sends $\mathcal G$ to the flat connection $G$ on $C$ uniquely defined
by the property that  $Fr_*G=\pi_*\CG$ as a $Fr_*D_C=\pi_*\CD$-module.

\subsection{A Reminder on the Chen--Zhu isomorphism}\label{schzh}\;

\begin{defn}
\label{defiso}
We define a morphism of algebraic stacks over the Hitchin base $A(C^{(1)})$:
\begin{equation} \label{fcdef}
    \widetilde{\Fc}
    : \CS\times_{A(C^{(1)})} \CM_{Dol}(C^{(1)},r)\to \CM_{dR}(C,r).
\end{equation}
Let $U$ be an $A(C^{(1)})$-scheme. An object in the groupoid $(\CS\times_{A(C^{(1)})} \CM_{Dol}(C^{(1)},r))(U)$ is a pair $(E, F)$ such that $E$ is a splitting of $\CD|_{Y_U}$ and $F$ is a Higgs bundle on $C_U^{(1)}$ with spectral curve $Y_U$.
Let $\CB(F)$ be the BNR sheaf of $F$ on $Y_U$ (notation as in (\ref{bnr})). We have that $E\otimes \CB(F)$ defines a $\CD|_{Y_U}$-module. Therefore it defines a flat connection $\widetilde{\Fc}(U)(E,F)$ on $C_U$ by the de Rham-BNR Theorem \ref{drbnr}.
Given isomorphisms between splittings and Higgs bundles $(E,F)\xrightarrow{\sim}(E',F')$, we obtain an isomorphism of $\CD|_{Y_U}$-modules $E\otimes\CB(F)\xrightarrow{\sim} E'\otimes \CB(F')$, thus a morphism of flat connections $\widetilde{\Fc}(U)(E,F)\xrightarrow{\sim}\widetilde{\Fc}(U)(E',F')$. 
\end{defn}

There is an action of the Picard stack $\CP^{st}_{Y/A}$ on the source of $\widetilde{\Fc}$:
\[\mathrm{act}: \CP^{st}_{Y/A}\times_{A(C^{(1)})} \CS\times_{A(C^{(1)})} \CM_{Dol}(C^{(1)},r) \longrightarrow \CS\times_{A(C^{(1)})} \CM_{Dol}(C^{(1)},r),\]
which, given a line bundle $L_U$ on $Y_U$, and a pair $(E,F)$ as above, the action $L_U\cdot(E,F)$ is defined to be $(E\otimes L_U^{-1}, L_U\cdot F)$, where $L_U\cdot F$ is the Higgs bundle on $C_U$ whose BNR sheaf is $L_U\otimes\CB(F)$. The action on the morphisms is defined in the natural way.

Consider the quotient of $\CS\times_{A(C^{(1)})} \CM_{Dol}(C^{(1)},r)$ by this action of $\CP^{st}_{Y/A}$. A priori the quotient is a 2-stack. However, since for every pair $(E,F)$ the automorphism group of the trivial line bundle on $Y_U$ injects into the automorphism group of $(E,F)$, by \cite[Lemme 4.7]{Ngo0}, we have that the quotient 2-stack is equivalent to a stack. We denote this stack by:
\[\CS\times^{\CP^{st}_{Y/A}} \CM_{Dol}(C^{(1)},r).\]

The following diagram is commutative:
\[\begin{tikzcd}
\CP^{st}_{Y/A}\times_{A(C^{(1)})} \CS\times_{A(C^{(1)})} \CM_{Dol}(C^{(1)},r)
\ar[r,shift left=.75ex,"act"]
  \ar[r,shift right=.75ex,swap,"pr_{2,3}"]
&
\CS\times_{A(C^{(1)})} \CM_{Dol}(C^{(1)},r) \ar[r,"\widetilde{\Fc}"] 
&
\CM_{dR}(C,r),
\end{tikzcd}\]
where $pr_{2,3}$ is the projection.
The universal property of the quotient stack \cite[p.419]{Ngo0} induces a natural morphism.
\[\Fc: \CS\times^{\CP^{st}_{Y/A}} \CM_{Dol}(C^{(1)},r)\to \CM_{dR}(C,r).\]

\begin{thm}\label{chzhthm}
\cite[Theorem 1.2]{ChZh}
$\Fc$ above is an isomorphism of algebraic stacks over $A(C^{(1)})$.
\end{thm}

Below in Theorem \ref{chzhgr} we point out that an isomorphism introduced in the second author's work  (see Theorem \ref{griso} below) gives rise to Chen--Zhu's isomorphism. First recall that the isomorphism in question is the following:
\begin{thm}\label{griso}
\cite[Theorem 3.29]{Groch}
There exists an isomorphism of algebraic stacks
\begin{equation}\label{Delta}
\Delta: \CS\times_{A(C^{(1)})}\CM_{Dol}(C^{(1)},r)\to \CS\times_{A(C^{(1)})} \CM_{dR}(C,r),
\end{equation}
where $\Delta(E,F)=(E,\widetilde{\Fc}(E,F))$.
\end{thm}

Given $(E,F)\in (\CS\times_{A(C^{(1)})}\CM_{Dol}(C^{(1)},r))(U)$, and $L\in \CP^{st}_{Y/A}(U)$, we have that 
\[\Delta(L\cdot(E,F))=\Delta(E\otimes L^{-1}, \widetilde{\Fc}(E\otimes L^{-1}, L\cdot F))=(E\otimes L^{-1}, \widetilde{\Fc}(E,F)).\]

Let $[\CS/\CP^{st}_{Y/A}]$ be the quotient 2-stack. By \cite[Lemme 2.7]{Ngo0} again, $[\CS/\CP^{st}_{Y/A}]$ is indeed a stack. Since for every object $E$ of $\CS(U)$, the set function $Aut(\CO_{Y_U})\cong Aut(E)$ is bijective, going through the proof of \textit{loc.cit.} we see that the $A(C^{(1)})$-stack $[\CS/\CP^{st}_{Y/A}]$ is indeed a sheaf which is furthermore isomorphic to $A(C^{(1)})$ itself.
Therefore, taking quotient by the Picard stack $\CP^{st}_{Y/A}$ of the isomorphism of Theorem \ref{griso}, and using that $[\CS/\CP^{st}_{Y/A}]\cong A(C^{(1)})$, we obtain Chen--Zhu's isomorphism:
\begin{thm}[Chen--Zhu]\label{chzhgr}
There exists an isomorphism of algebraic stacks over $A(C^{(1)}):$
\begin{equation}
\label{c0}
    \CS\times^{\CP^{st}_{Y/A}} \CM_{Dol}(C^{(1)},r)\cong[\CS/\CP^{st}_{Y/A}]\times_{A(C^{(1)})} \CM_{dR}(C,r)\cong \CM_{dR}(C,r).
\end{equation}
\end{thm}

\begin{rmk}
The above isomorphism establishes an a priori non-obvious fact: given an $A(C^{(1)})$-scheme $U$, if $\CS(U)=\emptyset$, then $\CM_{dR}(C,r)(U)=\emptyset$. Using the dR-BNR correspondence, we see that if $Y_U$ does not support a splitting of $\CD|_{Y_U}$, then it does not support a torsion free rank $p$ sheaf with a $\CD|_{Y_U}$-module structure either.
\end{rmk}

\subsection{Very good splittings preserve semistable loci}\;

Let $\CM_{Dol}^{ss}(C^{(1)},r)$ and $\CM_{dR}^{ss}(C,r)$ be the stack of  rank $r$ semistable Higgs bundles and flat connections.

The first main result of this paper is a semistable version of an isomorphism of stacks (\ref{c0}) introduced by Chen--Zhu and the second author.
This remedies \cite[Lemma 3.46]{Groch}, which is not correct as stated, see Remark \ref{mainmis}.

\begin{thm}\label{deltass}
The isomorphism of algebraic stacks $\Delta$ defined in (\ref{Delta}) restricts to an isomorphism of algebraic stacks: 
\begin{equation}
    \label{fcss}
    \Delta^{ss}: \CS^o\times_{A(C^{(1)})} \CM_{Dol}^{ss}(C^{(1)},r, d)\to \CS^o\times_{A(C^{(1)})} \CM_{dR}^{ss}(C,r, dp).
\end{equation}
\end{thm}
\begin{proof}

Let $U$ be an $A(C^{(1)})$-scheme together with a very good splitting $E$ on the spectral curve $Y_U.$ Let $F$ be a degree $d$ semistable Higgs bundle on $C^{(1)}_U=C^{(1)}\times_k U$ with spectral curve $Y_U$.
Since $\Delta(E,F)=(E,\widetilde{\Fc}(E,F))$ (c.f. (\ref{Delta})), 
with $\widetilde{\Fc}$ defined in Definition \ref{defiso}, i.e., the flat connection associated with the $\CD$-module $\CB(F)\otimes E$ via the de Rham-BNR correspondence as in Theorem \ref{drbnr}, 
we are reduced to showing that the flat connection $\widetilde{\Fc}(E, F)$ 
on $C_U$ is  semistable of degree  $dp.$

Since the degree conditions on semistability and very-goodness are checked fiberwise over $U$, we may assume that $U$ is a $k'$-point $a$ of $A(C^{(1)})$, where $k'/k$ is an algebraically closed field. Replacing $k$ with $k'$, we can assume that $U\in A(C^{(1)})(k)$. Let $Y_a$ be the corresponding spectral curve over $C^{(1)}$. Let $Y_a=\sum_i m_i\Gamma_i$, where $\Gamma_i$'s are pairwise distinct integral spectral curves  of repsective degrees
$r_i,$ and the  $m_i$'s are positive integers subject to $\sum_i r_i m_i=r.$


By Definition \ref{defiso}, the flat connection $\widetilde{\Fc}(E,F)=:G$ is the $D_C$-module that corresponds to the $\CD$-module $\CB(F)\otimes E$ on $Y_a$.
Using the notation in (\ref{4corners}),
we have an equality of $Fr_*D_C=\pi_*\CD$-modules: 
\begin{equation}
\label{frg}
    Fr_*G=\pi_*\Big(\CB(F)\otimes E\Big),
\end{equation}
and an equality of $\pi_*\CO_{T^*C^{(1)}}$-modules:
\begin{equation}
\label{fpb}
    F=\pi_*\CB(F).
\end{equation}

Given a Higgs subbundle $F'$ of $F$, we obtain an $\CO_{T^*C^{(1)}}$-submodule $\CB(F')$ via (\ref{fpb}), thus a sub-connection $G'$ of $G$ via (\ref{frg}). 
Conversely, given a sub-connection $G'$ of $G$, we obtain a $\CD|_{Y_a}$-submodule $K$ of $\CB(F)\otimes E$ via (\ref{frg}). 
The Morita equivalence \cite[Theorem 3.6.1]{Ginz} establishes an equivalence of categories
\begin{equation}
\label{morita}
\xymatrix{
 \CO_{Y_a}\text{-Mod}  \ar@/^1pc/[rrr]^-{-\otimes_{\CO_{Y_a}}E}_-{\cong} &&& \CD|_{Y_a}\text{-Mod} \ar@/^1pc/[lll]^-{End_{\CD_{Y_a}}(E,-)}_-{\cong}
    }
\end{equation}
Therefore, we have an isomorphism of $\CD|_{Y_a}$-modules $K\cong K'\otimes E$ for some $\CO_{Y_a}$-module $K'$. Consider the composition
\begin{equation}
\label{inc}
     i: K'\otimes E\xrightarrow{\sim} K\hookrightarrow \CB(F)\otimes E.
\end{equation}
Note that the Morita equivalence (\ref{morita}), being an equivalence between abelian categories, is automatically exact. Therefore, the inclusion of $\CD|_{Y_a}$-modules $i$ in (\ref{inc}) is induced by an inclusion of $\CO_{Y_a}$-modules $K'\hookrightarrow \CB(F)$. Therefore we have that $\pi_*K'=:F'$ is a Higgs subbundle of $F=\pi_*\CB(F)$. 
In summary, we have shown that there is a natural correspondence between Higgs subbundles $F'\subset F$ and flat sub-connections $G'\subset G$.

In what follows, we finish the proof by showing that $\mathrm{deg}(G')=p\,\mathrm{deg}(F')$, because then slopes of subobjects are rescaled by a factor of $p$, thus semistability is preserved under $\widetilde{\Fc}(E,-)$, therefore under $\Delta(E,-)$. 

Recall that in Definition \ref{defle}, we define $L_E$ to be the unique line bundle on $Y^p_a:=Y_a\times_{C^{(1)}} C$ subject to the condition that we have an identity of $W_*\CO_{Y_a^p}$-modules $W_*L_E= E.$ 
By Lemma \ref{vglb}, we have that $L_E\in \CP^{st,o}(Y^p_a)$. 

Let $Z_a$ be the subcurve of $Y_a$ that is the spectral curve for $F'$.
Let $E':=E|_{Z_a}$. 
Let $Z_a^p:=Z_a\times_{C^{(1)}} C$. Note that $Z_a^p$ is a subcurve of $Y_a^p$.
By Lemma \ref{lres} below, we have an isomorphism of line bundles on $Z_a^p$: $L_{E'}\cong(L_E)|_{Z_a^p}$.

Since $W:Z_a^p\to Z_a$ is a homeomorphism, by \cite[\href{https://stacks.math.columbia.edu/tag/0B55}{Tag 0B55}]{Sta22}, we have the following isomorphism of coherent sheaves on $Z_a$:
\begin{equation}
    \label{wle'}
    W_*(L_{E'}\otimes W^* \CB(F'))\cong E'\otimes \CB(F').
\end{equation}
Note that we have the following isomorphisms of vector bundles on $C^{(1)}$:
\[Fr_* \pi^p_*(L_{E'}\otimes W^*\CB(F'))\cong \pi_*W_*(L_{E'}\otimes W^*\CB(F'))\cong \pi_*(E'\otimes \CB(F')) \cong Fr_*(G'),\]
where the first isomorphism follows from the commutativity of the square (\ref{4corners}); the second isomorphism is given by (\ref{wle'}); the last isomorphism follows from the definition of $G'$, see Definition \ref{defiso}.
Therefore, we have an equality of degrees: (both sheaves are locally free on $C$)
\begin{equation}
\label{1eq}
    \mathrm{deg}_C(\pi^p_*(L_{E'}\otimes W^*\CB(F')))=\mathrm{deg}_C(G').
\end{equation}

By \cite[\S9, Proposition 5]{BLR} entails that for a line bundle $L$ on a nonreduced proper curve $n\Gamma$, we have that $\mathrm{deg}_{n\Gamma}(L)=n\mathrm{deg}_{\Gamma}(L|_{\Gamma})$. 

Therefore, by the very-goodness of the splitting $E$ (and here is where the very-goodness is used), we have that $L_{E'}\cong(L_E)|_{Z_a^p}\in \CP^{st,o}(Z_a^p)$,
so that by applying \cite[XIII, Theorem 4.6.(i),(ii)]{sga6}, we have an equality of Euler characteristics:
\begin{equation}
\label{chieq}
    \chi(L_{E'}\otimes W^*\CB(F'))) =\chi(W^*\CB(F')).
\end{equation}
Since the morphism $\pi^p: Z_a^p\to C$ is affine, the equality (\ref{chieq}) gives rise to the equality:
\begin{equation}
\label{chieqp}
    \chi(\pi^p_*(L_{E'}\otimes W^*\CB(F'))) =\chi(\pi^p_*W^*\CB(F')).
\end{equation}

Riemann-Roch and (\ref{chieqp}) entail that we have an equality of degrees:
\begin{equation}
\label{degeq}
    \mathrm{deg}_C(\pi^p_*(L_{E'}\otimes W^*\CB(F'))) =\mathrm{deg}_C(\pi^p_*W^*\CB(F')).
\end{equation}

Therefore, we have the following equality of degrees:
\begin{equation}
\label{2eq}
    \mathrm{deg}_C(\pi^p_*(L_{E'}\otimes W^*\CB(F'))) =\mathrm{deg}_C(\pi^p_*W^*\CB(F'))=\mathrm{deg}_C(Fr^* \pi_*\CB(F'))
    =\mathrm{deg}_C(Fr^*F')=p\mathrm{deg}_{C^{(1)}}(F'),
\end{equation}
where the first equality is (\ref{degeq}); the second follows from flat base change; the third follows from BNR; the last follows from the fact that the Frobenius morphism is of degree $p$.

Since the left-hand sides of the two equalities (\ref{1eq}) and (\ref{2eq}) are the same, we deduce the desired:
\[\mathrm{deg}_C(G')=p\,\mathrm{deg}_{C^{(1)}}(F').\]
\end{proof}

We have used the following lemma in the proof of Theorem \ref{deltass} above. 
It will be used again in the proof of Theorem \ref{vg is s} below.
\begin{lem}
\label{lres}
Let $Y/C^{(1)}$ be a spectral curve in $T^*C^{(1)}$. 
Let $Z/C^{(1)}$ be a spectral subcurve of $Y/C^{(1)}$.
Let $Y^p:=Y\times_{C}C^{(1)}$ and let $Z^p:=Z\times_C C^{(1)}$.
Let $E$ be a splitting of $\CD|_{Y}$.
Let $E'$ be $E|_Z$.
Then, in terms of the notation in Definition \ref{defle}, we have an isomorphism of line bundles on $Z^p$:
\[L_{E'}\cong (L_E)|_{Z^p}.\]
\end{lem}
\begin{proof}
We have an identity of $(W_*\CO_{Y^p})|_{Z}$-modules 
\begin{equation}
    \label{wle}
    \Big(W_*(L_{E})\Big)|_{Z}= E'.
\end{equation}
Since $W$ is affine, by base change and cohomology \cite[\href{https://stacks.math.columbia.edu/tag/02KG}{Tag 02KG}]{Sta22}, there is a natural identification of $\CO_{Y}$-algebras $(W_*\CO_{Y^p})|_{Z}\cong W_*\CO_{Z^p}$, and a natural identification of $\CO_{Y}$-modules $E'=E|_{Z}\cong W_*(L_E|_{Z^p})$. Under these natural identifications, the identity of $(W_*\CO_{Y^p})|_{Z}$-modules (\ref{wle}) becomes an isomorphism of $W_*\CO_{Z^p}$-modules:
\begin{equation}
    \label{wle2}
    W_*(L_E|_{Z^p})\cong E'.
\end{equation}
In view of Definition \ref{defle}, we obtain an isomorphism $L_{E'}\cong (L_E)|_{Z^p}$ of line bundles on $Z^p$.

\end{proof}

The action of $\CP^{st}_{Y/A}$ on $\CS\times_{A(C^{(1)})} \CM_{Dol}(C^{(1)},r)$ restricts to an action of $\CP^{st,o}_{Y/A}$ on $\CS^o\times_{A(C^{(1)})} \CM_{Dol}^{ss}(C^{(1)},r)$.

\begin{thm}[Chen--Zhu \& Groechenig, Semistable]\label{sseq}
There is a canonical isomorphism of stacks over $A(C^{(1)})$:
\begin{equation}
    \label{czgss}
    \Fc^{ss}: \CS^o\times^{\CP^{st,o}_{Y/A}}\CM_{Dol}^{ss}(C^{(1)},r,d)\cong \CM_{dR}^{ss}(C,r,dp).
\end{equation}
\end{thm}
\begin{proof}
By taking the quotient by the action of $\CP^{st,o}_{Y/A}$ on both domain and target of $\Delta^{ss}$, and using the same argument between Theorem \ref{griso} and Theorem \ref{chzhgr} above, we obtain 
\begin{equation}
    \CS^o\times^{\CP^{st,o}_{Y/A}}\CM_{Dol}^{ss}(C^{(1)},r,d)\cong [\CS^o/\CP_{Y/A}^{st,o}]\times_{A(C^{(1)})} \CM_{dR}^{ss}(C,r,dp)\cong \CM_{dR}^{ss}(C,r,dp).
\end{equation}
\end{proof}

\begin{rmk}
\label{mainmis}
In \cite[\S3.6, esp. Lemma 3.46]{Groch}, the author claims another version of Theorem \ref{deltass} above. There, instead of very good splittings $\CS^o$, the author used the stack of good splittings, $\CS^0=\CS\times_{\CP^{st}_{Y^p/A}}\CP^{st,0}_{Y^p/A}$, where $\CP^{st,0}_{Y^p/A}$ is the Picard stack of degree 0 line bundles on the spectral curves $Y^p/C$ relative to $A(C^{(1)})$. 
In other words, while in the definition of very good splittings we impose degree conditions on each component $m_i\Gamma_i$ of a fixed spectral curve $Y_a$ over $C^{(1)}$, in the definition of good splittings one just imposes a condition on the total degree on the whole curve $Y_a$. 
It was claimed that using good splittings, sub-Higgs-bundles of degree $d$ are mapped to subconnections of degree $pd$ under $\widetilde{\Fc}$ in (\ref{fcdef}), so that this claim would imply good splittings preserve semistability.

We caution the readers that the typographies of the good splittings $\CS^0$ and the very good splittings $\CS^o$ are very similar. 
However, this will not cause confusion in this paper because except in this remark, we never use the good splittings $\CS^0$ and only use the very good splittings $\CS^o.$

Theorem \ref{deltass} shows that the claim is valid if we consider very good splittings.
We now point out  that the claim made above concerning good splittings is not valid:

Let us take the rank 2 case. Assume good splittings preserve semistability. Let $Y_a$ be a spectral curve over $C^{(1)}$ with two smooth irreducible components $\Gamma_1$ and $\Gamma_2$. Let $E$ be a good splitting on $Y_a$. Let $F=F_1\oplus F_2$ be a polystable Higgs bundle with proper sub-Higgs-bundles $F_1$ and $F_2$ on $C^{(1)}$, and with BNR sheaf $\CB(F)$ a line bundle on $Y_a$. 
By assumption, the $\CD_{Y_a}$-module $\CB(F)\otimes E$ gives rise to a polystable flat connection $G=G_1\oplus G_2$ on $C$ with $\mathrm{deg}(G_i)=p\mathrm{deg}(F_i)$. 
Let $L$ be a line bundle on $Y_a$ with multidegree $(n,-n)$ for some $n\ne 0$. Since $E$ is a good splitting, we have that $E\otimes L$ is also a good splitting. 
Let $G'=G_1'\oplus G_2'$ be the flat connection corresponding to the $\CD$-module $\CB(F=F_1 \oplus F_2)\otimes L\otimes E$. 
We have that $\mathrm{deg}_{\Gamma_1}(\CB(F)\otimes E\otimes L)= \mathrm{deg}_{\Gamma_1}(\CB(F)\otimes E)+pn$.
A Riemann-Roch calculation then shows that 
\[deg(G_1')= deg(G_1)+pn=pdeg(F_1)+pn.\]



The inaccuracy in the proof in \textit{loc.cit.} is in the very last sentence ``Using one more time that the fibrewise degree function is locally constant in flat families, we deduce that $\mathrm{deg}_{Y(k)(1)_a} i^*S = (1-p)(1-h)k$ is satisfied for every geometric point of $V$ , hence also $x$, and therefore every geometric point of $U$." Here the setup is the following: $U$ and $V$ are \'etale neighborhoods of $A(C^{(1)})$ over which $\CS$ is trivialized; the image of $U$ does not contain the origin, while the image of $V$ does. The point $x$ is in the image of the fiber product $U\times_{A}V$. The point $a$ is in the image of $V$. This sentence seems to assume that every splitting over $Y_{U\times_A V}$ can be extended to one over $Y_U,$ and  this is not true. 

In short, goodness does not restrict to subcurves, whereas very goodness does.
\end{rmk}

\subsection{Very good splittings and the de Rham moduli space}\;

A splitting $E$ of $\CS^o$ on a spectral curve $Y_b$ is tautologically a torsion free rank $p$ sheaf with a $\CD$-module structure on $Y_b$.
Therefore, by the de Rham-BNR Theorem \ref{drbnr}, the splitting $E$ corresponds to a
rank $r$ flat connection on $C.$
In this way, we obtain a monomorphism of $A(C^{(1)})$-algebraic stacks $\CS^o\hookrightarrow \CM_{dR}(C)$.
Below, we show that this inclusion is actually an open immersion into the stable part of $\CM_{dR}(C)$.

Just as in Proposition \ref{picost}.(2), we set $d^o:=r(r-1)(1-g)$.
\begin{thm}
\label{vg is s}
The natural inclusion $\CS^o\hookrightarrow \CM_{dR}(C)$ factors through an open immersion:
\[\CS^o\hookrightarrow \CM_{dR}^{s}(C,r, pd^o),\]
where the latter is the moduli stack of stable flat connections of rank $r$ and degree $pd^o$.
\end{thm}
\begin{proof}
We first show that the monomorphism $\CS^o\hookrightarrow \CM_{dR}(C)$ factors through the open immersion $\CM_{dR}^{s}(C,r,pd^o)\hookrightarrow\CM_{dR}(C)$.
Given an $A(C^{(1)})$-scheme $U$ and a $U$-family of flat connections $(G,\nabla)$, to check that $(G,\nabla)$ is stable with a particular degree, we just need to check it over each geometric point of $U$. 
Therefore, we can check the factorization $\CS^o\hookrightarrow \CM_{dR}^{s}(C,r,pd^o)\hookrightarrow \CM_{dR}(C)$ over geometric points of $A(C^{(1)})$.

We fix $b\in A(C^{(1)})(k')$ for some algebraically closed field $k'/k$.
Replacing $k$ with $k'$, we can assume that $k'=k$.
The corresponding spectral curve $Y_b$ can be written as $\sum_i m_i\Gamma_i,$ where $\Gamma_i$'s are pairwise distinct integral spectral curves of degree $r_i$ over $C^{(1)}$.
Let $E$ be a splitting of $\CD|_{Y_b}$.
Let $G$ be the flat connection on $C$ that corresponds to $E$ on $Y_b$
under the de Rham-BNR correspondence as in Theorem \ref{drbnr}.
We first show that:

\textbf{Claim 1:} $G$ is stable of degree $pd^o$.

Let $G_1$ be a subconnection of $G.$ 
The quotient sheaf $G_2:=G/G_1$ inherits the flat connection on $G$.
Recall that stability can be checked on locally free quotients of positive rank.
To show Claim 1, it suffices to show:

\textbf{Claim $\mathbf{1'}$:} for every torsion free quotient connection $G_2$ of $G$ with $0 < \mathrm{rank}(G_2)=r'\le r,$ we have that $\mathrm{deg}_C(G_2)= r'(r'-1)p(1-g)$.

Indeed, if the degrees are of the form given above, then for any quotient $G_2$ with rank $r'<r$, we have that $\mu(G_2)=(r'-1)p(1-g)>(r-1)p(1-g)=\mu(G),$ and the desired stability of $G$ follows.

Proof of Claim $1'$.
Let $E_2$ be the dR-BNR sheaves of $G_2$ as in Theorem \ref{drbnr}. In particular, $E_2$ is a torsion free rank $p$ sheaf on a degree $r'$ spectral curve $Y_b'$.
By the description of the dR-BNR correspondence as in \S\ref{secdrbnr}, we have that the quotient of flat connections $G\twoheadrightarrow G_2$ corresponds to a quotient of $\CD$-modules $E\twoheadrightarrow E_2$.
As a consequence, we have that $Y_b'$ is a subcurve of $Y_b$, i.e., we have that $Y_b'=\sum_i m_i'\Gamma_i$, with $0\le m_i'\le m_i$ for all $i$.

Before continuing the proof of Claim $1'$, we prove:

\textbf{Claim 2:} $E_2\cong E|_{Y_b'}$.

Indeed, since $E_2$ has support on $Y_b'$, the quotient morphism of $\CO_{T^*C^{(1)}}$-modules $E\twoheadrightarrow E_2$ factors into the morpshisms of $\CO_{T^*C^{(1)}}$-modules $E\twoheadrightarrow E|_{Y_b'}\twoheadrightarrow E_2.$
Let $K:=Ker(E|_{Y_b'}\twoheadrightarrow E_2)$.
Let $\eta_j$ be the generic point of $\Gamma_j.$
Since $E$ is locally free of rank $p$, we have that $l_{\CO_{Y_b',\eta}}(E_{\eta_k})=m_i'p$.
On the other hand, in the above paragraph we have seen that $E_2$ is torsion free of rank $p$ on $Y_b'$. In view of our definition of rank in \S\ref{rkdeg}, and using that
length is additive \cite[\href{https://stacks.math.columbia.edu/tag/00IV}{Tag 00IV}]{Sta22}, we have that $K_{\eta_j}=0$ for each $j$.
Since $E_2$ is locally free, we have that $K=0$. Therefore the quotient morphism $E|_{Y_b'}\twoheadrightarrow E_2$ is indeed an isomorphism $E|_{Y_b'}\cong E_2$.
We have thus proved \textbf{Claim 2} above.

We now continue to show Claim $1'$. Recall that in Definition \ref{defle}, and we are using the notation therein, we define a line bundle $L_E$ on $Y_b^p=\sum_i m_i\Gamma_i^p\subset T^{1,*}C$ such that there is an equality of $W_*\CO_{T^{1,*}C}$-modules $E=W_*L_E$. 
Corollary \ref{vglb} shows that $L_E\in \CP^{st,o}(Y_b^p)$.
By \cite[\S9, Proposition 5]{BMR}, we have that the line bundle $(L_E)|_{\sum_i m_i'\Gamma_i^p}$ is in $\CP^{st,o}(\sum_i m_i'\Gamma_i^p)$.
By Lemma \ref{lres} above, we have an isomorphism of line bundles $(L_E)|_{\sum_i m_i'\Gamma_i^p}\cong L_{E_2}$. 
Therefore, we have that $L_{E_2}\in \CP^{st,o}(\sum_i m_i'\Gamma_i^p)$.
Note that we have an isomorphism of $\pi_*W_*\CO_{T^{*,1}C}=Fr_*\pi^p_*\mathcal{O}_{T^{1,*}C}$-modules (notation as in (\ref{4corners})):
\[Fr_*\pi_*^p L_{E_2}\cong \pi_* W_*L_{E_2}= \pi_* E_2= Fr_* G_2.\]
In view of the natural inclusion $\mathcal{O}_C\to \pi_*^p\CO_{T^{1,*}C}$, we have that the above isomorphisms are also isomorphisms of $Fr_*\CO_C$-modules.
Taking the first and last terms, we obtain an isomorphism $\pi^p_* L_{E_2}\cong G_2$ of vector bundles on $C$. Therefore, we can calculate the degree of $G_2$ using Riemann-Roch:
\begin{equation}
\label{xlh}
    \chi(L_{E_2})=0+\chi(\CO_{\sum_i m_i'\Gamma_i^p})=deg(G_2)+r'\chi(\CO_C).
\end{equation}
Note that we have $\chi(\CO_{\sum_i m_i'\Gamma_i^p})=r'\chi(\CO_C)+r'(r'-1)p(1-g)$. Plug this value into (\ref{xlh}), we immediately obtain the degree of $G_2$ as in Claim $1'$.

Therefore, \textbf{Claim }$\mathbf{1'}$, thus \textbf{Claim 1}, is proved.
Claim 1 implies that we have a monomorphism of algebraic stacks $\CS^o\hookrightarrow \CM_{dR}^{s}(C,r,pd^o)$.

To finish the proof, it remains to show that the monomorphism is indeed an open immersion. The basic idea is to reduce to the fact that $\CP^{st,o}_{Y/A}$ is an open algebraic substack of $M^s_{Dol}(C^{(1)})$. 

Let $T$ be a $\CM_{dR}$-scheme.
Let $T^o$ be the fiber product $T\times_{\CM_{dR}}\CS^o$.
The goal is to show that the natural morphism $T^o\hookrightarrow T$ is an open immersion of schemes.

By Lemma \ref{monost}, we have that $T^o\to T$ is a universally injective morphism of schemes.
Since a universally injective \'etale morphism of schemes is an open immersion \cite[\href{https://stacks.math.columbia.edu/tag/02LC}{Tag 02LC}]{Sta22}, we are left with showing that $T^o\hookrightarrow T$ is \'etale. 
By \cite[\href{https://stacks.math.columbia.edu/tag/02HM}{Tag 02HM}]{Sta22}, we need to show that $T^o\hookrightarrow T$ is locally of finite presentation and formally \'etale. 
Since both $\CS^o$ and $\CM_{dR}^s(r(r-1)p(1-g))$ are of finite presentation over $A(C^{(1)})$, by \cite[\href{https://stacks.math.columbia.edu/tag/06Q6}{Tag 06Q6}]{Sta22}, we have that $\CS^o\hookrightarrow \CM_{dR}$ is of finite presentation. Therefore the base change $T^o\hookrightarrow T$ is also of finite presentation \cite[\href{https://stacks.math.columbia.edu/tag/06Q4}{Tag 06Q4}]{Sta22}.
We are reduced to showing that $T^o\to T$ is formally \'etale. 

Let $V$ be an affine $T$-scheme and let $V\hookrightarrow V'$ be a first-order thickening of affine $T$-schemes.
Given the solid arrows as below, we need to show that there exists a unique dotted diagonal arrow $\alpha$ making the diagram commutative:
\begin{equation}
\label{vts}
    \xymatrix{
    V\ar[r] \ar[d] & T^o \ar[r] \ar[d] & \CS^o\ar[d]\\
    V' \ar[r] \ar@{-->}[ur]^-{\alpha} \ar@{-->}[urr]_-(.76){\beta} & T \ar[r] & \CM_{dR}^s(r(r-1)p(1-g)).
    }
\end{equation}
Since the right square is Cartesian, we only need to show that there exists a unique dotted arrow $\beta$ making the above diagram commutative. 

By Noetherian approximation, we can assume that $T$ is Noetherian. 
Therefore, by \cite[\href{https://stacks.math.columbia.edu/tag/02HT}{Tag 02HT}]{Sta22}, 
we can assume that $V'$ is the spectrum of an artinian $k$-algebra. 

Note that we can now replace the Hitchin base $A(C^{(1)})$ with $V'$, i.e., let subscript $(\cdot)_{V'}$ denote the base change of an $A(C^{(1)})$-scheme along the composition $V'\to \CM_{dR}^s\to A(C^{(1)})$. We need to find the dashed arrow $\beta_{V'}$ in the following commutative diagram:
\begin{equation}
    \label{}
    \xymatrix{
    V \ar[r] \ar[d] & \CS^o_{V'} \ar[d]\\
    V'\ar[r] \ar@{-->}[ur]^-{\beta_{V'}} & \CM_{dR,V'}^s.
    }
\end{equation}
However, by \cite[Lemma 3.23]{Groch}, we have that $\CS^o_{V'}$ is just the trivial $\CP^{o,st}_{Y/A,V'}$-torsor. Thus Corollary \ref{sseq} implies that we have an isomorphism of stacks $\CM_{dR,V'}^s(C, pd^o)\cong \CM_{Dol,V'}^s(C^{(1)},d^o)$, fitting into the following commutative diagram:
\begin{equation}
\label{}
    \xymatrix{
    V \ar[r] \ar[d] & \CS^o_{V'} \ar[d]  & \CP^{st,o}_{Y/A,V'}\ar[d] \ar[l]_{\sim}\\
    V'\ar[r] \ar@{-->}[ur]^-{\beta_{V'}} \ar@{-->}[urr]_-(.76){\gamma_{V'}} & \CM_{dR,V'}^s  & \CM_{Dol,V'}^s \ar[l]_{\sim}.
    }
\end{equation}
Since the right vertical arrow is an open immersion (Proposition \ref{picost}.(2)) we know that the dashed arrow $\gamma_{V'}$ above exsits. Thus the desired morphism $\beta_{V'}$ exists. 
\end{proof}

Langer's \cite[Theorem 1.1]{Langer14} gives  the existence of quasi-projective moduli spaces 
 $M_{dR}^{ss}(C, r, pd)$ (resp. $M_{dR}^{s}(C, r, pd)$) of semistable
(resp. stable) flat connections of rank $r$ and degree $pd.$ 
We state his result in the langauge of stacks.
In order to do so, we first review briefly some of Simpson's related results:
 In view of the boundedness result \cite[Theorem 0.2]{Langer04}, Langer's construction of $M_{dR}^{ss}$ coincides with Simpson's, which is a GIT quotient of the form $Q\sslash SL_N,$ see \cite[Theorem 4.7]{SimpI}. 
Simpson has an alternative description of this moduli space as a GIT quotient of the form $R_{dR}\sslash GL_r,$ where $R_{dR}$ is the de Rham representation space, see \cite[Theorem 4.10]{SimpI}.
In fact, \cite[Lemma 4.9]{SimpI} implies that the representation space $R_{dR}$, when viewed as a stack in setoids, is equivalent to the stack of flat connections with a framing at a fixed closed point of $C$.
Therefore, we have an isomorphism of stacks $\CM_{dR}^{ss}\cong [R_{dR}/GL_r]$, see also \cite[p.11,12]{Simp95}. 
Finally, by Alper's \cite[Theorem 9.1.4]{Alper-ad}, we see that the composition $\CM_{dR}^{ss}\cong [R_{dR}/GL_r]\to R\sslash GL_r\cong M_{dR}^{ss}$ is an adequate moduli space.

\begin{thm}\cite[Theorem 1.1]{Langer14}
\label{langer}
The moduli stack $\CM_{dR}^{ss}(C,r,pd)$ ($\CM_{dR}^{s}(C,r,pd),$resp.)  admits an adequate moduli space $M_{dR}^{ss}(C,r,pd)$
($M_{dR}^{s}(C,r,pd),$ resp.)
which is a quasi-projective scheme over $k.$
Furthermore, in the stable case,  the adequate moduli space is in fact tame and  a $\mathbb{G}_m$-gerbe. Moreover,  the open subscheme
$M_{dR}^{s}(C,r,pd)$  of $M_{dR}^{ss}(C,r,pd)$ is smooth over $k.$
\end{thm}

By recalling that $d^o=r(r-1)(1-g)$,
Theorem \ref{vg is s} and Theorem \ref{langer}  imply that the stack of very good splittings admits a tame moduli space:

\begin{thm}
\label{vgsch}
The stack $\CS^o$ of very good splittings is a $\BG_m$ gerbe over an open subscheme $S^o$ of $M_{dR}^s(C,r,pd^o)$, fitting into a Cartesian diagram of stacks:
\begin{equation}
\label{ssq}
    \xymatrix{
    \CS^o\ar@{^{(}->}[r] \ar[d]& \CM_{dR}^s(C,r,pd^o)\ar[d]\\
    S^o \ar@{^{(}->}[r] & M_{dR}^s(C,r,pd^o).
    }
\end{equation}
We have that $S^o$ is a good and tame moduli space.

Furthermore, let $P^o$ be the group scheme representing the functor $\CP_{Y/A(C^{(1)})}^{sh,o}$. The $A(C^{(1)})$-scheme $S^o$ is a torsor under $P^o$.
In particular, we have that $S^o$ is fiberwise connected over $A(C^{(1)}).$
\end{thm}
\begin{proof}
By Theorem \ref{langer}, we have that $\CM_{dR}^s$ is a $\BG_m$-gerbe over $M_{dR}^s$.
Therefore, we have an identification of topological spaces $|\CM_{dR}^s|= |M_{dR}^s|$, see \cite[\href{https://stacks.math.columbia.edu/tag/06R9}{Tag 06R9}]{Sta22}. 
By Theorem \ref{vg is s}, we have that $\CS^o\hookrightarrow\CM_{dR}^s(C,r,pd^o)$ is an open immersion.
Therefore, we have an open immersion of topological spaces $|\CS^o|\hookrightarrow|\CM_{dR}^s|= |M_{dR}^s|$.
We define $S^o$ to be the open subscheme whose underlying topological space is the image of $|\CS^o|$ inside $|M_{dR}^s|$.

By general properties of $\BG_m$-gerbes as in \cite[Proposition 2.2.1.6, Lemma 2.1.1.13]{Lieblich}, a quasi-coherent sheaf $\CF$ on $\CS^o$ has a natural decomposition with respect to the $\BG_m$-weights $\CF=\bigoplus\CF_n,$ and taking the pushforward by $\CS^o\to S^o$ is just taking the weight zero part. 
Therefore, we have that the morphism $\CS^o\to S^o$ satisfies the two defining properties of the good moduli space morphism \cite[Definition 4.1]{Alper}.
Furthermore, by general property of gerbes \cite[\href{https://stacks.math.columbia.edu/tag/06R9}{Tag 06R9}]{Sta22}, we have that $|\CS^o|\to |S^o|$ is a universal homeomorphism. 
Therefore, $\CS^o\to S^o$ is also a tame moduli space morphism as in \cite[Definition 7.1]{Alper}.

Therefore, we see that the morphism $\CS^o\to M_{dR}^s$ factors through a surjective morphism $\CS^o\to S^o$, and that $\CS^o\to S^o$ is a $\BG_m$-gerbe.

The natural morphism $\CS^o\to S^o\times_{M_{dR}^s}\CM_{dR}^s$ is a morphism of $\BG_m$-gerbes which is bijective on automorphism groups, thus it must be an isomorphism by \cite[Lemma 4.6]{OV}. Therefore, the square (\ref{ssq}) is Cartesian.

Finally, Lemma \ref{poso} shows that $\CS^o$ is a torsor under $\CP^{st,o}_{Y/A}$. 
The universal property \cite[Theorem 4.16.(vi)]{Alper} of tame moduli spaces then induces an action of $P^o$ on $S^o$.
Since $\CP^{st,o}_{Y/A}$ (resp. $\CS^o$) is a $\BG_m$-gerbe over $P^o$ (resp. $S^o$), we see that $S^o$ is also a torsor under $P^o$.
\end{proof}

\begin{cor}
There is an isomorphism of quasi-projective $A(C^{(1)})$-schemes
\begin{equation}
    S^o\times_{A(C^{(1)})} M_{Dol}^{ss}(C^{(1)},r,d)\cong S^o\times_{A(C^{(1)})} M_{dR}^{ss}(C,r,dp).
\end{equation}
\end{cor}
\begin{proof}
By Theorem \ref{vgsch}, every stack in Lemma \ref{deltass} has a moduli space which is representable by a quasi-projective scheme. 
Taking moduli space commutes with finite products.
Therefore, replacing the moduli stacks in Lemma \ref{deltass} gives the corollary.
\end{proof}

Analogously to Corollary \ref{sseq}, we obtain:
\begin{thm}
\label{schois}
There is an isomorphism of quasi-projective $A(C^{(1)})$-schemes
\begin{equation}
\label{schemeiso}
    S^o\times^{P^o}M_{Dol}^{ss}(C^{(1)},r,d)\cong M_{dR}^{ss}(C,r,dp).
\end{equation}
\end{thm}

\subsection{An analogous result over finite fields}$\;$

Assume that $C/k$ is the base change of an $\mathbb{F}_q$-curve $\mathcal{C}$. Then, all moduli stacks considered above are defined over $\mathbb{F}_q$ and the de Rham and Dolbeault-Hitchin morphisms are equally obtained via base change from $\mathbb{F}_q$-morphisms.

The constructions of this article carry over to the finite field case, \textit{mutatis mutandis}. The corollary below shows how the geometric properties of the stack of very good splittings can be exploited to compare the number of rational points in de Rham and Dolbeault Hitchin fibers.
\begin{cor}\label{fifield}
Under the aforementioned assumption, let $a \in A(\mathcal{C}^{(1)})(\mathbb{F}_q)$ be an $\mathbb{F}_q$-rational point. Then, there is an isomorphism of $\mathbb{F}_q$-stacks
$$h^{-1}_{Dol}(a) \cong h^{-1}_{dR}(a),$$
respecting the semistable locus. In particular, both Hitchin fibers have the same number of $\mathbb{F}_q$-rational points.
\end{cor}
\begin{proof}
The fiberwise version of Theorem \ref{sseq} obtained by base change yields an isomorphism:
$$ S^o_a\times^{\mathcal{P}^{st,o}_a} h^{-1}_{Dol}(a)  \cong h^{-1}_{dR}(a)$$
preserving the semistable locus.

The fibre $S^o_a$ is a $\mathcal{P}_a^{st,o}$-torsor over $\mathbb{F}_q$. Since the latter is a connected group scheme, Lang's theorem implies that the torsor is trivial. This implies the assertion.
\end{proof}

\section{Isomorphic Decomposition Theorems}

In this section, we prove Theorem \ref{isoic} as a direct consequence of 
Theorem \ref{isodec} which, in turn is proved as an application of
Theorem \ref{schois} and of a homotopy-type lemma for actions of group schemes
with geometrically connected fibers.

Let us remark that even in the coprime case, where the moduli spaces are smooth,
the Ng\^o Support Theorem \cite[Theorem 7.2.1]{Ngo} \cite[Theorem 7.0.3]{dCRS},
which leverages the action of the group scheme $P^o,$
is not sufficient to reach the desired conclusion (\ref{icsh}) because
the local systems appearing in the Ng\^o strings are governed by the top direct image sheaves
$R^{\rm top} h_{\rm dR,*} \overline{\mathbb Q}_\ell$
and $R^{\rm top} h_{\rm *} \overline{\mathbb Q}_\ell$, which a priori
may 
differ for the de Rham-Hitchin and for the Hitchin morphism; they only become isomorphic after \'etale base changes of $A(C^{(1)}).$

Theorem \ref{isoic} can be viewed as yet another incarnation of a NAHT
in positive characterisitc, where the moduli spaces are not necessarily 
in natural bijection, but their cohomology groups are isomorphic.

\begin{thm}
\label{isodec}
We have the following two canonical isomorphisms between the perverse cohomology sheaves in $D_c^b(A(C^{(1)}),\oql)$:
\begin{equation}
\label{doldr}
    \pcs^*(h_{Dol,*}\oql)\cong \pcs^*(h_{dR,*}\oql),\quad \pcs^*(h_{Dol,*}\CI\CC)\cong \pcs^*(h_{dR,*}\CI\CC).
\end{equation}
\end{thm}
\begin{proof}
Let $h^S: S^o\times^{P^o} M_{Dol}^{ss}(C^{(1)})\to A(C^{(1)})$ be the twisted Hitchin fibration. 
By the isomorphism (\ref{schemeiso}), we have the following isomorphism of perverse cohomology sheaves in $D_c^b(A(C^{(1)}),\oql)$:
\begin{equation}
\label{sdr}
    \pcs^*(h_{*}^S\oql)\cong \pcs^*(h_{dR,*}\oql).
\end{equation}
Below we show that $\pcs^*(h_{*}^S\oql)\cong \pcs^*(h_{Dol,*}\oql):$

Let $\{ U_{\alpha}\}_{\alpha}$ be an \'etale covering of $A(C^{(1)})$ so that for each $\alpha$, the $P^o|_{U_{\alpha}}$-torsor $S^o|_{U_{\alpha}}$ is trivialized. 
We fix such trivializations $\widetilde{\psi_{\alpha}}: S^o|_{U_{\alpha}} \xrightarrow{\sim} P^o|_{U_{\alpha}}$.

In view of the isomorphism (\ref{schemeiso}), each trivialization $\widetilde{\psi_{\alpha}}$ induces an isomorphism in $D_c^b(U_{\alpha},\oql)$:
\begin{equation}
\label{psia}
    \psi_{\alpha}: \pcs^*(h^S_*\oql)|_{U_{\alpha}}\xrightarrow{\sim} \pcs^*(h_{Dol,*}\oql)|_{U_{\alpha}}.
\end{equation}

For each pair of indices $\alpha$ and $\beta$, let $U_{\alpha\beta}$ be the fiber product $U_{\alpha}\times_{A(C^{(1)})} U_{\beta}$.

\textbf{Claim:}  We have the following commutative diagram in $D_c^b(U_{\alpha\beta},\oql):$
\begin{equation}
\label{claimcomm}
    \xymatrix{
    (\pcs^*(h^S_*\oql)|_{U_{\alpha}})|_{U_{\alpha\beta}} \ar[r]^-{\psi_{\alpha}|_{U_{\alpha\beta}}}_-{\sim} 
    \ar[d]_-{can}
    & (\pcs^*(h_{Dol,*}\oql)|_{U_{\alpha}})|_{U_{\alpha\beta}}
    \ar[d]^-{can}\\
    (\pcs^*(h^S_*\oql)|_{U_{\beta}})|_{U_{\alpha\beta}} \ar[r]^-{\psi_{\beta}|_{U_{\alpha\beta}}}_-{\sim} 
    & (\pcs^*(h_{Dol,*}\oql)|_{U_{\beta}})|_{U_{\alpha\beta}},
    }
\end{equation}
where the vertical morphisms are the canonical isomorphisms induced by the pullbacks to the fiber product $U_{\alpha\beta}$.

We first show the Claim above.
Throughout the proof we fix the canonical isomorphisms of the form $(X|_{U_{\alpha}})|_{U_{\alpha\beta}}\xrightarrow{\sim} X|_{U_{\alpha\beta}}$ where $X$ is an object in $D_c^b(A(C^{(1)}),\oql)$. We hence regard those canonical isomorphisms, such as the vertical isomorphisms $can$ in (\ref{claimcomm}), as identities. 
To show the Claim, we thus need to show that \[\psi_{\beta\alpha}:=(\psi_{\beta}|_{U_{\alpha\beta}})^{-1}\circ \psi_{\alpha}|_{U_{\alpha\beta}}= \mathrm{id}\in \mathrm{Aut}_{D_c^b(U_{\alpha\beta})}(\pcs^*(h_*^S\oql)|_{U_{\alpha\beta}}).\]

The isomorphism $\psi_{\beta\alpha}$ above is induced by the isomorphism of the $U_{\alpha\beta}$-scheme
\[\widetilde{\psi_{\beta\alpha}}:= (\widetilde{\psi_{\beta}}|_{U_{\alpha\beta}})^{-1}\circ \widetilde{\psi_{\alpha}}|_{U_{\alpha\beta}}\in  \mathrm{Aut}_{U_{\alpha\beta}-\mathrm{sch}}(S^o|_{U_{\alpha\beta}})\cong P^o(U_{\alpha\beta}),\]
where the last isomorphism of automorphism groups is due to the fact that $ S^o$ is a torsor under $ P^o$. 
This shows that the isomorphism $\psi_{\beta\alpha}$ is induced by the action of a global section $\widetilde{\psi_{\beta\alpha}}$ of $ P^o|_{U_{\alpha\beta}}$ on the scheme $( S^o\times^{ P^o} M_{Dol})|_{U_{\alpha\beta}}$. 

Since the group scheme $ P^o$ is smooth and has connected geometric fibers over $A(C^{(1)})$, the homotopy lemma \cite[Lemma 3.2.3]{LauNgo} shows that any global section of $ P^o$ acts trivially on the perverse cohomology sheaves of any $ P^o$-schemes. 
Therefore we have that $\psi_{\beta\alpha}$ is the identity on $\pcs^*(h_*^S\oql)|_{U_{\alpha\beta}}$. Thus the \textbf{Claim} above is shown. 

By \cite[Proposition 3.2.2.]{BBD}, the presheaf that sends any \'etale covering $U\to A(C^{(1)})$ to the group $End_{D_b^c(U,\oql)}(\pcs^*(h^S_*\oql)|_{U})$ in fact defines a sheaf. Therefore, by the Claim above, we see that the isomorphisms $\psi_{\alpha}$'s can be glued together to an isomorphism 
\begin{equation}
\label{sdol}
    \psi: \pcs^*(h_*^S\oql)\xrightarrow{\sim} \pcs^*(h_{Dol,*}\oql).
\end{equation}

Combining the isomorphisms (\ref{sdr}) and (\ref{sdol}), we obtain the desired first isomorphism in (\ref{doldr}).
Replacing the constant sheaves $\oql$ with the corresponding IC sheaves $\CI\CC$, the exact same argument above gives the second isomorphism in (\ref{doldr}).

It remains to show that the isomorphisms in (\ref{doldr}) is canonical. 
Given any two trivializing covers, we can always take the fiber product of the two, and we are reduce to show that given a fixed trivilizing cover, the isomorphism (\ref{doldr}) is independent of the choice of trivializations $\widetilde{\psi_{\alpha}}$'s. 
Indeed, different trivializations differ by the action of a section of $P^o$, but we have seen above that the action of such sections on the perverse cohomology sheaves are trivial. Thus the proof is finished. 
\end{proof}

\begin{thm}
\label{isoic}
There are distinguished isomorphisms in $D_c^b(A(C^{(1)}),\oql)$:
\begin{equation}
\label{icsh}
    h_{Dol,*}\CI\CC\cong h_{dR,*}\CI\CC.
\end{equation}
In particular, we have isomorphisms of intersection cohomology groups:
\begin{equation}
\label{icc}
    IH^*(M_{Dol}^{ss}(C^{(1)},d),\oql)\cong IH^*(M_{dR}^{ss}(C,dp),\oql).
\end{equation}
\end{thm}
\begin{proof}
The isomorphism (\ref{icsh}) follows from the second isomorphism in (\ref{doldr}) and the Decomposition Theorem.
\end{proof}

\begin{rmk}
\label{dspl}
Choose relatively ample line bundles on $M_{Dol}/A(C^{(1)})$ and $(S^o\times^{P^o}M_{Dol})/A(C^{(1)})$.
Given these choices, there are three distinguished Deligne splittings in the derived category of the form $ \bigoplus_*\pcs^*(h_*\CI\CC)[-*]\cong h_*\CI\CC$ which can be taken as a canonical isomorphism in the Decomposition Theorem for the direct image complexes for the Hitchin and de Rham-Hitchin morphisms, respectively. 
By using such splittings as in \cite[\S2.4]{dC13}, together with the canonical isomorphism (\ref{doldr}), we obtain distinguished isomorphism (\ref{icsh}) $h_{Dol,*}\CI\CC\cong h_{dR,*}\CI\CC$.
The isomorphisms 
(\ref{icc}) follow by taking cohomology.
It is also possible to obtain two additional distinguished isomorphisms (\ref{icc})  by working directly in intersection cohomology, i.e. by by-passing the production of isomorphisms (\ref{icsh}); see \cite[Thm 1.1.1]{dC13}.
\end{rmk}

\begin{rmk} Using Corollary \ref{fifield}, when we are in the case $(d,r)=(dp,r)=1$, an isomorphism (\ref{doldr}) can also be obtained by using, in the context of the function-sheaf dictionary,  the smoothness and purity of the moduli spaces as well as  point-counting over finite fields. See also \cite[Remark 3.14.(i)]{dCZ}.
\end{rmk}

\end{document}